\documentclass{amsart}

\usepackage{amssymb,stmaryrd}
\usepackage{amsmath}
\usepackage{graphicx}

\newtheorem{theorem}{Theorem}[section]

\newtheorem{lemma}[theorem]{Lemma}
\newtheorem{proposition}[theorem]{Proposition}
\newtheorem{corollary}[theorem]{Corollary}

\theoremstyle{definition}
\newtheorem{definition}[theorem]{Definition}

\theoremstyle{remark}

\begin{document}

\title[Harmonic Besov Spaces with Small Exponents]{ Harmonic Besov Spaces with small exponents}

\author{\"{O}mer Faruk Do\u{g}an}
\address{Department of Mathematics,  Tek$\dot{\hbox{\i}}$rda\u{g} Namik Kemal University, Namik Kemal Mahalles$\dot{\hbox{\i}}$, Kamp\"{u}s caddes$\dot{\hbox{\i}}$ No:1,
59030 S\"{u}leymanpa\c{s}a/Tek$\dot{\hbox{\i}}$rda\u{g}, Turkey}
\email{ofdogan@nku.edu.tr}



\begin{abstract}
We study harmonic Besov spaces $b^p_\alpha$  on the unit ball of $\mathbb{R}^n$, where $0<p<1$ and $\alpha\in\mathbb{R}$. We provide characterizations in terms of partial and radial derivatives and certain radial differential operators that are more compatible with reproducing
kernels of harmonic Bergman-Besov spaces. We show that the dual of harmonic Besov space $b^p_\alpha$  is weighted Bloch space $b_\beta^{\infty}$  under certain volume integral pairing for $0<p<1$ and $\alpha,\beta\in\mathbb{R}$. Our other results are about growth at the boundary and atomic decomposition.
\end{abstract}

\date{\today}

\maketitle

\section{Introduction}\label{s-introduction}

Let $\mathbb{B}$ be the open unit ball and $\mathbb{S}$ be the unit sphere in $\mathbb{R}^n$  for $n\geq 2$. We denote the normalized Lebesgue volume measure on $\mathbb{B}$ by $d\nu$ normalized as $\nu(\mathbb{B})=1$. For $\alpha\in \mathbb{R}$, we define on $\mathbb{B}$ the weighted volume measures $d\nu_\alpha$  by
\[
d\nu_\alpha(x)=\frac{1}{V_\alpha} (1-|x|^2)^\alpha d\nu(x).
\]
These measures are finite when $\alpha>-1$ and in this case we choose $V_\alpha$ so that $\nu_\alpha(\mathbb{B})=1$. For $\alpha\leq -1$, we set $V_\alpha=1$.  We denote the Lebesgue classes with respect to $\nu_\alpha$ by $L^p_\alpha=L^p(d\nu_\alpha)$, $0<p<\infty$. For any $f \in L^p_\alpha$, we write
\begin{equation*}
\|f\|_{L^{p}_\alpha} = \left(\int_{\mathbb{B}} |f|^p d\nu_{\alpha} \right)^{1/p}.
\end{equation*}
Note that when $1\leq p<\infty$ the space $L^p_\alpha$ is a Banach space with the above norm. When $0<p<1$, the space $L^p_\alpha$ is a quasi-Banach space; i.e. it is a complete metric space with metric
\begin{equation*}
d(f,g)=\|f-g\|_{L^{p}_\alpha}^{p}
\end{equation*}
which satisfies the properties $d(f,g)=d(f-g,0)$ and $d(\lambda f, 0)=|\lambda|^{p}d(f,0)$ for $\lambda\in \mathbb{C}$.

In this paper we consider two-parameter family of harmonic Besov spaces $b^p_\alpha$ with $0<p<1$ and $\alpha \in \mathbb{R}$ which are sometimes called Bergman-Sobolev spaces or Bergman-Besov spaces. We study in a detailed and systematic way the properties of this family. For the case $1\leq p<\infty$ and $\alpha \in \mathbb{R}$  the spaces $b^p_\alpha$ are studied in   \cite{GKU1,GKU2}. The holomorphic counterpart of this family of spaces for the full range of parameters  have been studied  in \cite{ZZ}.

Let $h(\mathbb{B})$ be the space of all complex-valued harmonic functions on $\mathbb{B}$ with the topology of uniform convergence on compact
subsets. The weighted harmonic Bergman space $b^p_\alpha$ for $\alpha>-1$ and $0<p<\infty$ is $b^p_\alpha=h(\mathbb{B}) \cap L^p_\alpha$. When $\alpha=0$ we denote the ordinary unweighted harmonic Bergman spaces by $b^p$. The space $b^2_\alpha$ $(\alpha> -1)$ is a reproducing kernel Hilbert space  with reproducing kernel $R_\alpha(x,y)$.  When $p \geq 1$, the space $b^p_\alpha$ is a Banach space, while for $p < 1 $ the norm inherited from $L^{p}_\alpha$
defines only a quasinorm, under which $b^p_{\alpha}$ is complete.

The aim of this work is to extend the definition of $b^p_\alpha$ when $0<p<1$ to the case in which $\alpha$ is any real number and develop a theory for the extended family of spaces. This family of spaces are already extended to all $\alpha\in \mathbb{R}$ for the case $1\leq p<\infty$ in \cite{GKU1, GKU2} where also the reproducing kernels $R_\alpha(x,y)$ are extended to the whole range $\alpha\in \mathbb{R}$. We will give a review of these in Section
\ref{kernels-dst}

To extend the definition of $b^p_\alpha$ to the range $\alpha \leq -1$, we need to consider growth rates of derivatives of $u\in h(\mathbb{B})$. To describe the functions in Besov spaces, we will employ three different types of differentiation. For a multi-index $m=(m_1,\dots,m_n)$ where $m_1,\dots,m_n$ are non-negative integers
\begin{equation*}
 \partial^m u= \frac{\partial^{|m|} u}{\partial x_1^{m_1}\cdots\partial x_n^{m_n}},
\end{equation*}
is the usual partial derivative for smooth $u$, where $|m|=m_1+\dots+m_n$.

For any $u\in h(\mathbb{B})$ there exist homogeneous harmonic polynomials $u_k$ of degree $k$ such that
$u=\sum_{k=0}^{\infty} u_k$,  the series converging absolutely and uniformly  on compact subsets of $\mathbb{B}$ which is called the homogeneous expansion of $u$ (see \cite{ABR}). The radial derivative $\mathcal{R} u$ of $u\in h(\mathbb{B})$ is defined as
\begin{equation}\label{Radial-Derivative}
  \mathcal{R}u (x)=x\cdot \nabla u(x) = \sum_{k=0}^{\infty} k u_k (x),
\end{equation}
in which $\nabla$ denotes the usual gradient. More generally, we define
\begin{equation*}
  \mathcal{R}^N u(x) = \mathcal{R} \mathcal{R}^{N-1} u(x) = \sum_{k=0}^\infty k^N u_k(x), \quad N=2,3, \ldots .
\end{equation*}

In addition to partial and radial derivatives we will extensively use certain radial differential operators $D^t_s$, $(s,t \in \mathbb{R})$ introduced in \cite{GKU1} and \cite{GKU2}. These operators are defined in terms of reproducing kernels of harmonic Besov spaces and are specific to these spaces, but still mapping  $h(\mathbb{B})$ onto itself. Thus using the operators $D^t_s$ instead of  partial or radial derivatives is more advantageous. The properties of $D^t_s$ will be reviewed in Section \ref{kernels-dst}. We only note for now that $t$ determines the order of the  differentiation and $s$ plays a minor role.

The following theorem will enable us to define harmonic Besov space $b^p_\alpha$ when $0<p<1$ for the whole range $\alpha \in \mathbb{R}$. Here, $\mathbb{N}=\{0,1,2,\ldots\}$ denotes the set of natural numbers with $0$ included.

\begin{theorem}\label{Theorem-Equiv-Bloch}
Let $0<p<1$,  $\alpha \in \mathbb{R}$ and $u\in h(\mathbb{B})$. The following are equivalent:
\begin{enumerate}
  \item[(a)] For every $N\in \mathbb{N}$ with $\alpha+pN>-1$, we have $(1-|x|^2)^N \partial^m u \in L^p_{\alpha}$ for every multi-index $m$ with $|m|=N$.
  \item[(b)] There exists an $N\in \mathbb{N}$ with $\alpha+pN>-1$ such that $(1-|x|^2)^N \partial^m u \in L^p_{\alpha}$ for every multi-index $m$ with $|m|=N$.
  \item[(c)] For every $N\in \mathbb{N}$ with $\alpha+pN>-1$, we have $(1-|x|^2)^N \mathcal{R}^N u \in L^p_{\alpha}$.
  \item[(d)] There exists an $N\in \mathbb{N}$ with $\alpha+pN>-1$ such that $(1-|x|^2)^N \mathcal{R}^N u \in L^p_{\alpha}$.
  \item[(e)] For every $s,t \in \mathbb{R}$ with $\alpha+pt>-1$, we have $(1-|x|^2)^t D_s^t u \in L^p_{\alpha}$.
  \item[(f)] There exist $s,t \in \mathbb{R}$ with $\alpha+pt>-1$ such that $(1-|x|^2)^t D_s^t u \in L^p_{\alpha}$.
\end{enumerate}
Moreover, if $\alpha+pN>-1$ and $\alpha+pt>-1$, then
\begin{equation}\label{Norm}
\begin{split}
\| (1-|x|^2)^t D^t_s u \|_{L^p_\alpha} &\sim |u(0)| + \| (1-|x|^2)^N \mathcal{R}^N u \|_{L^p_\alpha} \\
&\sim \sum_{|m| \leq N-1}|(\partial^m u)(0)| + \sum_{|m|=N} \| (1-|x|^2)^N \partial^m u \|_{L^p_\alpha}.
\end{split}
\end{equation}
\end{theorem}

For the case $1\leq p<\infty$, the above theorem is proved in \cite[Theorem 1.2]{GKU2}. Thus these theorems complete the picture about the interchangeability of various kinds of derivatives in defining the two-parameter harmonic Besov space family for the full range of parameters.  For the holomorphic analogues   with $0< p<\infty$  see \cite{KHZ}  when $\alpha >-1$ and see \cite{ZZ} when $\alpha \in \mathbb{R}$ . Also our presentation here  is influenced by the methods in that \cite{ZZ,KHZ}, but the technical details are different.

\begin{definition}\label{DefinitionAllAlpha}
Let $0<p<1$ and $\alpha \in \mathbb{R}$. The harmonic Besov space $b^{p}_{\alpha}$ consists of those $u\in h(\mathbb{B})$ such that any one of the equivalent conditions of Theorem \ref{Theorem-Equiv-Bloch} is satisfied.
\end{definition}

The space $b^{p}_{\alpha}$ do not depend on the choice of $N$ or $s,t$ as long as $\alpha+pN>-1$ and $\alpha+pt>-1$ is satisfied. When
$\alpha>-1$ ($\alpha=0$), one can choose N = 0 in part (b) of the above theorem and observe that Definition \ref{DefinitionAllAlpha} is consistent with the definition of weighted(unweighted) harmonic Bergman spaces. Notice that we have the same definition for harmonic Besov spaces $b^{p}_{\alpha}$ when $1\leq p<\infty$, see \cite{GKU2} where it is also shown that the choice of $N$ or $s,t$ is irrelevant
as long as $\alpha+pN>-1$ and $\alpha+pt>-1$  is satisfied.

We mention a few immediate consequences of Definition \ref{DefinitionAllAlpha}. First, if $u\in h(\overline{\mathbb{B}})$, then $u\in b^{p}_{\alpha}$ for each $0<p<1$ and $\alpha \in \mathbb{R}$; in particular every $b^{p}_{\alpha}$ is non-trivial since it clearly contains harmonic polynomials. Also the following simple inclusion property holds:
\begin{equation}\label{Inclusion}
 b^{p}_\alpha \subset b^{p}_{\beta} \qquad (\text{for\ } \alpha < \beta).
\end{equation}

The inclusion relations between harmonic besov spaces not only for $0<p<1$ but also for the full range of parameters is completely
determineted by Theorem 1.1 and Theorem 1.2 in \cite{DU2}. 

When $0<p<1$ and $\alpha>-1$ we have a standard quasinorm on $b^{p}_\alpha$, but when $\alpha \leq -1$ we do not. For $\alpha\in \mathbb{R}$, if we pick any $N\in\mathbb{N}$ with $\alpha+pN>-1$ or pick $s,t\in \mathbb{R}$ with $\alpha+pt>-1$, each term in (\ref{Norm}) is a quasinorm on $b^{p}_\alpha$. Since all these quasinorms are equivalent, there is no essential difference in choosing any one of them; and we will denote any one of these quasinorms by $\| \cdot \|_{b^{p}_\alpha}$ without indicating the dependence on $N$ or $s,t$.

For $s,t \in \mathbb{R}$ consider the linear transformations $I^t_s$ defined for $u\in h(\mathbb{B})$ by
\begin{equation*}
  I^t_s u(x) := (1-|x|^2)^t D^t_s u(x).
\end{equation*}
It is clear from Theorem \ref{Theorem-Equiv-Bloch} that given $\alpha\in \mathbb{R}$, if $t$ is chosen to satisfy $\alpha+pt>-1$, then $u\in b^{p}_\alpha$ if and only if $I^t_s u \in L^p_\alpha$ and $\| I^t_s u \|_{L^p_\alpha}$ is a quasinorm on $b^{p}_\alpha$.

Roughly speaking the $``p=\infty"$ case of Besov spaces $b^p_\alpha$ is the one parameter family of Bloch spaces $b^\infty_\alpha$. When $\alpha=0$, the well-known harmonic Bloch space $b^\infty_0$ is the space of all $u\in h(\mathbb{B})$ such that
\[
 \sup_{x\in\mathbb{B}}\, (1-|x|^2)|\nabla u(x)| <\infty.
\]
 When $\alpha>0$,  the  weighted harmonic Bloch space $b^\infty_\alpha$  is defined by
\[
b^\infty_\alpha=\Big\{u\in h(\mathbb{B}): \sup_{x\in \mathbb{B}}\, (1-|x|^2)^\alpha |u(x)|<\infty\Big\} \qquad (\alpha>0).
\] Let $\alpha\in \mathbb{R}$. Pick a non-negative integer $N$ such that
\begin{equation}\label{alpha-N}
\alpha+N>0.
\end{equation}
The weighted harmonic Bloch space $b^\infty_\alpha$ consists of all $u\in h(\mathbb{B})$ such that
\[
\sup_{x\in \mathbb{B}}\, (1-|x|^2)^{\alpha+N} |\partial^m u(x)|<\infty,
\]
for every multi-index $m$ with $|m|=N$. As before the spaces $b^\infty_\alpha$ do not depend on the choice of $N$ as long as (\ref{alpha-N}) is satisfied and partial derivatives can be replaced with radial derivatives or the operators $D^t_s$. These are studied in detail in \cite{DU1}.

It is well-known that for $1< p < \infty$ and $\alpha>-1$, $(b^p_\alpha)'$, the dual space of the harmonic Bergman space $b^p_\alpha$  can be identified with $b^{p'}_\alpha$, where $1/p+1/p' =1$. It is shown in \cite[Theorem 13.4]{GKU2} that this statement is true for all $\alpha\in \mathbb{R}$. In \cite{DU1} by extending this result to the case $p=1$, it is proved that  $(b^1_\alpha)'$ can be identified with $b^{\infty}_\beta$ for any $\alpha,\beta \in \mathbb{R}$.

The second aim of this paper is to consider remaining cases, namely identifying the dual spaces of $b^p_\alpha$ when $0<p<1$ and $\alpha\in \mathbb{R}$.

The harmonic Besov space $b^p_\alpha$ is not a Banach space when $0<p<1$. However, we can still
consider its dual space. In fact, we define the dual space of $b^p_\alpha$ for $0<p<1$ in exactly the same way as we do for $p\geq 1$. Thus $(b^p_\alpha)'$ is the space of all bounded linear functionals on $b^p_\alpha$. Then $(b^p_\alpha)'$ becomes a Banach space with the norm
\[
\|F\|=\sup\Big\{|F(u)|: \|u\|_{b^{p}_\alpha}\leq 1\Big\}
\]
for any bounded linear functional $ F$ on $b^p_\alpha$.

Our main result is the following theorem. Here, $\alpha,\beta \in \mathbb{R}$ without any restriction and the aforomentioned identification can be obtained using many different pairings. More precisely, we have the following.

\begin{theorem}\label{Theorem-Dual-of-b1q}
Let $0<p\leq1$, $\alpha\in \mathbb{R}$ and $\rho=(n+\alpha)/p-n$. Pick $s,t$ such that
\begin{align}
s &> \rho, \label{Dual-S}\\
\alpha+pt &> -1. \label{Dual-T}
\end{align}
The dual of $b^p_\alpha$ can be identified with $b^{\infty}_\beta$ (for any $\beta \in \mathbb{R}$) under the pairing
\begin{equation}\label{Dual-Pairing}
\langle u, v \rangle = \int_{\mathbb{B}} I^t_s u \, \overline{I^{s-\rho-\beta}_{t+\rho+\beta}v} \ d\nu_{\rho+\beta}, \qquad (u\in b^p_\alpha, \ v\in b^{\infty}_\beta).
\end{equation}
\end{theorem}

For  $\alpha>-1$ and  $\beta=0$ the above theorem is proved  in \cite{R}, where a slightly different pairing (involving a limit) is used. For the holomorphic analogues of this theorem with $\alpha>-1$, $\beta=0$ see \cite{ KHZ} and with $\alpha,\beta \in \mathbb{R}$  see \cite{ ZZ}. Once again, the pairings in \cite{ZZ} are slightly different than our pairings and involve a limit.

When $\alpha>-1$, atomic decomposition of harmonic Bergman spaces $b^{p}_{\alpha}$ is proved in \cite{CR}. In the following theorem we  extend their result to all $\alpha\in \mathbb{R}$.
\begin{theorem}\label{atomicbesov}
Let $0< p<1$ and $\alpha\in \mathbb{R}$. There exists a sequence $(x_m)$ of points of $\mathbb{B}$ with the following property: Let $s>(n+\alpha)/p-n$.
 \begin{enumerate}
    \item[(i)] For every  $u\in b^{p}_{\alpha}$, there exists $(\lambda_{m})\in \ell^{p}$ such that
\begin{equation}\label{atbesove}
u(x)=\sum_{m=1}^{\infty} \lambda_{m}(1-|x_m|^{2})^{n+s-(n+\alpha)/p} R_{s}(x,x_m)
\end{equation}
and  $\|\lambda_{m}\|_{\ell^{p}}\lesssim \|u\|_{b^{p}_{\alpha}}$.
    \item[(ii)]  For every $(\lambda_{m})\in \ell^{p}$, the function $u$ defined in (\ref{atbesove}) is in $b^{p}_{\alpha}$ and
 $\|u\|_{b^{p}_{\alpha}} \lesssim\|\lambda_{m}\|_{\ell^{p}}$.
  \end{enumerate}
\end{theorem}

The paper is organized as follows. In Section \ref{s-preliminaries} we collect some known facts which we will
heavily use in later sections. In Section \ref{Section-subharmonic} we show that product of two harmonic functions have subharmonic behaviour and this leads to an integral inequality which will be needed in the proofs of the main results. In Section \ref{Section-Proof-of-T1} we will prove Theorem \ref{Theorem-Equiv-Bloch} and derive basic properties of the space $b^{p}_\alpha$ for $0<p<1$ and $\alpha\in \mathbb{R}$ . Finally in Section \ref{pe-dual-atomic} using results from the previous section, we will obtain some pointwise
estimates for functions in $b^{p}_\alpha$. Then we will show that the dual of Besov space $b^p_\alpha$, $0<p<1$ can be identified with $b^{\infty}_\beta$ (for any $\beta \in \mathbb{R}$) under suitable pairings. We will  also obtain atomic decomposition for all $\alpha \in \mathbb{R}$.

\section{Preliminaries}\label{s-preliminaries}

 For two positive expressions $X$ and $Y$, we write $X\lesssim Y$ if there exists a positive constant $C$, whose exact value is inessential, such that $X\leq CY$. We also write $X\sim Y$ if both  $X\lesssim Y$ and $Y\lesssim X$ hold.

The Pochhammer symbol $(a)_b$ is defined by
\[
(a)_b=\frac{\Gamma(a+b)}{\Gamma(a)},
\]
when $a$ and $a+b$ are off the pole set $-\mathbb{N}$ of the gamma function. Stirling formula gives
\begin{equation}\label{Stirling}
\frac{(a)_c}{(b)_c} \sim c^{a-b}, \quad c\to\infty.
\end{equation}

For $x,y\in \mathbb{B}$, we will use the notation
\begin{equation*}
  [x,y]=\sqrt{1-2 x\cdot y + |x|^2 |y|^2},
\end{equation*}
where $x\cdot y$ denotes the inner product of $x$ and $y$ in $\mathbb{R}^n$. It is elementary to show that the equalities
\begin{equation*}
  [x,y] = \Big| |y|x - \frac{y}{|y|}\Big| = \Big| |x|y - \frac{x}{|x|}\Big|,
\end{equation*}
hold for every nonzero $x,y$. Note that $0< 1-|x||y|\leq [x,y]\leq 1+|x||y|<2$ for $x,y \in \mathbb{B}$ and when $y=\zeta\in \mathbb{S}$, we have $[x,\zeta]=|x-\zeta|$.

\subsection{Zonal harmonics}\label{Subsection-Zonal Harmonics} We denote the space of all homogeneous harmonic polynomials on $\mathbb{R}^n$ of degree $k$ by $\mathcal{H}_k(\mathbb{R}^n)$. The restriction of $u_k\in \mathcal{H}_k(\mathbb{R}^n)$ to the unit sphere $\mathbb{S}$ is called a spherical harmonic and we denote the space of spherical harmonics of degree $k$ by $\mathcal{H}_k(\mathbb{S})$. The finite-dimensional space $\mathcal{H}_k(\mathbb{S})$ $\subset L^2(\mathbb{S})$ is a reproducing kernel Hilbert space: For $\zeta\in \mathbb{S}$, there exists (real-valued) $Z_k(\cdot, \zeta)$ such that
\begin{equation*}
  u_k(\zeta)=\int_{\mathbb{S}} u_k(\eta) Z_k(\eta,\zeta) d\sigma(\eta) \quad (\forall u_k\in \mathcal{H}_k(\mathbb{S})),
\end{equation*}
where $d\sigma$ is normalized surface area measure on $\mathbb{S}$. The spherical harmonic $Z_k(\cdot, \zeta)$ is called zonal harmonic of degree $k$ with pole $\zeta$. It can be extended to $\mathbb{R}^n\times \mathbb{R}^n$ by making it homogeneous in each variable: If $x=|x|\eta$, $y=|y|\zeta$ with $\eta,\zeta\in \mathbb{S}$,
\begin{equation*}
  Z_k(x,y)=|x|^k |y|^k Z_k(\eta,\zeta),\quad k=1,2,\ldots
\end{equation*}
For $k=0$, we set $Z_0(x,y)\equiv 1$. For future reference we state the following properties of $Z_k$ (see Chapter 5 of \cite{ABR} for details).
\begin{lemma}\label{Lemma-Zk}
The following properties hold:
\begin{enumerate}
\item[(a)] $Z_k(x,y)$ is real-valued and symmetric in its variables.
\item[(b)] $Z_k(x,0)=Z_k(0,y)=0$, for every $x,y\in \mathbb{R}^n,\ k=1,2,\ldots$
\item[(c)] For $k\geq 1$ and $\zeta\in \mathbb{S}$, $\max_{\eta\in \mathbb{S}} |Z_k(\eta,\zeta)|= Z_k(\zeta,\zeta)$ and $Z_k(\zeta,\zeta) \sim k^{n-2}$. Therefore $|Z_k(x,y)| \lesssim |x|^k |y|^k k^{n-2}$.
\item[(d)] If $Y^{j}_{k} \, (j=1,2,\ldots,J_{k})$ is an orthonormal basis of $\mathcal{H}_k(S)$, then $Z_k(\eta,\zeta)= \sum_{j=1}^{J_{k}} Y^{j}_{k}(\eta)\overline{Y^{j}_{k}(\zeta)}$.
\item[(e)] If $u_k \in \mathcal{H}_k(\mathbb{R}^n)$, then $u_k(x) = \int_{\mathbb{S}} u_k(\eta) Z_k(x,\eta) d\sigma(\eta)$.
\item[(f)] If $u_k \in \mathcal{H}_k(\mathbb{R}^n)$ and $l \neq k$, then $\int_{\mathbb{S}} u_k(\eta) Z_l(x,\eta) d\sigma(\eta)=0$.
\end{enumerate}
\end{lemma}

\subsection{Reproducing Kernels and the Operators $D^t_s$}\label{kernels-dst}
It is well-known that the weighted harmonic Bergman spaces $b^{2}_\alpha$ for  $\alpha>-1$ is a reproducing kernel Hilbert space with reproducing
kernel $R_\alpha(x,y)$:
\begin{equation}\label{Reproducing-Bergman}
u(x)=\int_{\mathbb{B}} u(y) R_\alpha(x,y)\, d\nu_\alpha(y) \qquad \forall u\in b^2_\alpha, \ \forall x\in \mathbb{B} \quad (\alpha>-1).
\end{equation}
The homogeneous expansion of $R_\alpha(x,y)$ can be expressed in terms of zonal harmonics (see \cite{DS}, \cite{M})
\[
R_\alpha(x,y)=\sum_{k=0}^\infty \frac{(1+n/2+\alpha)_k}{(n/2)_k} Z_k(x,y)=:\sum_{k=0}^\infty \gamma_k(\alpha) Z_k(x,y), \qquad (\alpha>-1),
\]
where the series absolutely and uniformly converges on $K\times \mathbb{B}$, for any compact subset $K$ of $\mathbb{B}$. $R_{\alpha}(x,y)$ is real-valued, symmetric in the variables $x$ and $y$ and harmonic with respect to each variable since the same is true
for all $Z_k(x,y)$. Notice that the coefficients $\gamma_k(\alpha)$ make sense as long as $\alpha>-(1+n/2)$, and they satisfy
\begin{equation}\label{gamma_k}
\gamma_k(\alpha)\sim k^{\alpha+1} \quad (k\to \infty),
\end{equation}
for all such $k$ by (\ref{Stirling}).

The reproducing kernels $R_\alpha(x,y)$ can be extended to all $\alpha\in \mathbb{R}$ (see \cite{GKU1,GKU2}), where the crucial point is not the precise form of the kernel but preserving the property (\ref{gamma_k}).

\begin{definition}\label{Rq - Series expansion}
Let $\alpha\in\mathbb{R}$. Define
\[
\gamma_k(\alpha):= \begin{cases}
\dfrac{(1+n/2+\alpha)_k}{(n/2)_k}, &\text{if $\, \alpha>-(1+n/2)$}; \\
\noalign{\medskip}
\dfrac{((1)_k)^2}{(1-(n/2+\alpha))_k (n/2)_k}, &\text{if $\, \alpha\leq -(1+n/2)$};
\end{cases}
\]
and
\begin{equation}\label{Rq - Series expansion1}
\displaystyle R_\alpha(x,y):=\sum_{k=0}^\infty \gamma_k(\alpha) Z_k(x,y).
\end{equation}
\end{definition}
Checking the two cases above by (\ref{Stirling}), the property (\ref{gamma_k}) holds for all $\alpha\in\mathbb{R}$.
$R_\alpha(x,y)$ given as in Definition \ref{Rq - Series expansion} is a reproducing kernel and generates the reproducing
kernel Hilbert space $b^{2}_\alpha$ on $ \mathbb{B } $ for all $\alpha\in \mathbb{R}$ (see \cite{GKU2}).

For every $\alpha\in \mathbb{R}$ we have $\gamma_{0} (\alpha)=1$ and therefore with Lemma \ref{Lemma-Zk} (b),
\begin{equation}\label{Rq(x,0)}
R_\alpha(x,0)=R_\alpha(0,y)=1, \quad \forall x,y\in \mathbb{B} \quad (\forall \alpha\in \mathbb{R}).
\end{equation}

 $R_\alpha(x,y)$ is harmonic as a function of either of its
variables  on $\overline{\mathbb{B}}$ and if $K\subset \mathbb{B}$ is compact and $m$ is a multi-index
\begin{equation}\label{Rq-uniformly bounded}
| \partial^m R_\alpha(x,y) | \lesssim 1, \quad \forall x\in K, \, y\in \overline{\mathbb{B}},
\end{equation}
where differentiation is performed in the first variable.

The radial differential operator $D^t_s$ acts as a coefficient multiplier on the homogeneous expansion of a harmonic function and is defined in the
following way (see \cite{GKU1, GKU2}).

\begin{definition}
Let $u=\sum_{k=0}^\infty u_k\in h(\mathbb{B})$ be given by its homogeneous expansion. For $s,t\in\mathbb{R}$, we define radial differential operators $D_s^t : h(\mathbb{B}) \to h(\mathbb{B})$ by
\begin{equation}\label{Define-Dst}
  D_s^t u := \sum_{k=0}^\infty \frac{\gamma_k(s+t)}{\gamma_k(s)} \, u_k.
\end{equation}
\end{definition}

 By (\ref{gamma_k}), $\gamma_k(s+t)/\gamma_k(s) \sim k^t$ and therefore roughly speaking $D^t_s$ multiplies the $k^{th}$ homogeneous part of $u$ by $k^t$.
  When $t>0$ the operator $D^t_s$ acts as a differential operator and when $t<0$ as an integral operator. For every $s\in \mathbb{R}$, $D_s^0=I$, the identity. The parameter $s$ plays a minor role and is used to have the precise relation
\begin{equation}\label{Dst - Rs}
D_s^t R_s(x,y)=R_{s+t}(x,y),
\end{equation}
where differentiation is performed on either of the variables $x$ or $y$ and by symmetry
it does not matter which. Compared to partial or radial derivatives an important property of $D^t_s$ is that it is invertible with two-sided inverse $D_{s+t}^{-t}$:
\begin{equation}\label{inverse of Dst}
D^{-t}_{s+t} D^t_s = D^t_s D^{-t}_{s+t} = I,
\end{equation}
which follows from the additive property
\begin{equation}\label{Additive-Dst}
D_{s+t}^{z} D_s^t = D_s^{z+t}.
\end{equation}

 The following lemma is Theorem 3.2 of \cite{GKU2}
 \begin{lemma}\label{Lemma-cont-Dst}
 Equip $h(\mathbb{B})$ with the topology of uniform convergence on compact subsets. Then for every $s,t \in \mathbb{R}$, the map $D^t_s: h(\mathbb{B})\to h(\mathbb{B})$ is continuous.
\end{lemma}

 We can push $D^t_s$ into some certain integrals. The following Lemma and the corollary after that (the case $c=s$ ) are taken from \cite{DU1}.

\begin{lemma}\label{Lemma-Push-Dst}
Let $c\in \mathbb{R}$ and $f\in L_c^1$. For every $s,t \in \mathbb{R}$,
\begin{equation*}
D^t_s \int_{\mathbb{B}} R_c(x,y) f(y) d\nu_c(y) = \int_{\mathbb{B}} D^t_s R_c(x,y) f(y) d\nu_c(y).
\end{equation*}
\end{lemma}

\begin{corollary}\label{Corollary-Push-Dst}
  Let $s \in \mathbb{R}$ and $f\in L^1_s$. For every $t\in \mathbb{R}$,
\begin{equation*}
  D^t_s \int_{\mathbb{B}} R_s(x,y) f(y) d\nu_s(y) = \int_{\mathbb{B}} R_{s+t}(x,y) f(y) d\nu_s(y).
\end{equation*}
\end{corollary}

In some cases we can write $D^t_s$ as an integral operator. It is well known that the reproducing formula (\ref{Reproducing-Bergman}) remains true for all $u\in b^1_\alpha$ $(\alpha>-1)$. Therefore if we apply $D^t_s$ to both sides of (\ref{Reproducing-Bergman}) and use the above corollary we obtain following result easily. See also \cite{DU1}.
\begin{corollary}\label{Corollary-Dst-Integral}
Let $s>-1$ and $u\in L^1_s \cap h(\mathbb{B})$. For every $t\in \mathbb{R}$,
\begin{equation}\label{Dst - Integral operator}
  D^t_s u(x) = \int_{\mathbb{B}} R_{s+t}(x,y) u(y) d\nu_s(y).
\end{equation}
\end{corollary}
The operator $D^t_s$ as an integral operator as in (\ref{Dst - Integral operator}) appears in \cite{JP}.

Analogous to Theorem \ref{Theorem-Equiv-Bloch}, the spaces $b^p_\alpha$ for $1\leq p < \infty$  can equivalently be defined by using the operators $D^t_s$. Given $1\leq p < \infty$ and $\alpha\in\mathbb{R}$, pick $s,t\in\mathbb{R}$ such that $\alpha+pt>-1$. The harmonic Besov space $b^p_\alpha$ consists of all $u\in h(\mathbb{B})$ such that (see Theorem 1.2 of \cite{GKU2})
\[
\|u\|^p_{b^p_\alpha} = \|I^t_s u  \|^p_{L^p_\alpha}=\frac{1}{V_\alpha} \int_{\mathbb{B}} |D^t_s u(x)|^p (1-|x|^2)^{\alpha+pt} d\nu(x) <\infty.
\]
Strictly speaking, the norm depends on $s$ and $t$, but this is not mentioned as it is known that every choice of the pair $(s,t)$ leads to an equivalent norm.

For $1\leq p < \infty$, we have bounded harmonic projections from the $L^p_\alpha$ onto the $b^p_\alpha$ which provide integral representations
for elements of  $b^p_\alpha$.
\begin{definition}
  For $s\in \mathbb{R}$, the Bergman-Besov projection is
\begin{equation*}
  Q_s f(x) = \int_{\mathbb{B}} R_s(x,y) f(y) d\nu_s(y),
\end{equation*}
for suitable $f$.
\end{definition}

The following projection theorem for $b^p_\alpha$ spaces is Theorem 1.5 of \cite{GKU2}.

\begin{theorem}[\cite{GKU2}]\label{Theorem-Projection-Besov}
Let $1\leq p < \infty$ and $\alpha, s \in \mathbb{R}$. Then $Q_s : L^p_\alpha \to b^p_\alpha$ is bounded (and onto) if and only if
\begin{equation}\label{Projection-Besov-S}
\alpha+1< p(s+1).
\end{equation}
Given an $s$ satisfying (\ref{Projection-Besov-S}) if $t$ satisfies
\begin{equation}\label{Projection-T}
\alpha+pt>-1,
\end{equation}
then for $u\in b^p_\alpha$, we have
\begin{equation}\label{QsIst=f}
Q_s I^t_s u = \frac{V_{s+t}}{V_s} u.
\end{equation}
\end{theorem}

By (\ref{QsIst=f}) we have the following integral representation: For $u\in b^{p}_\alpha$, if (\ref{Projection-Besov-S}) and (\ref{Projection-T}) holds, then
\begin{equation}\label{Reproducing}
u(x)=\frac{V_s}{V_{s+t}}\int_{\mathbb{B}} R_s(x,y) I^t_s u(y) \, d\nu_s(y) = \int_{\mathbb{B}} R_s(x,y) D^t_s u(y) \, d\nu_{s+t}(y).
\end{equation}
The integral representation is very useful and many properties of the spaces $b^p_\alpha$ when $1\leq p < \infty$ are extracted from this in \cite{GKU2}.

\subsection{Estimates of Reproducing Kernels}
In this subsection, we present some  properties and estimates on the derivatives of reproducing kernels  $R_\alpha(x,y)$  which we need in later sections. The reproducing kernels $R_\alpha(x,y)$ are well-studied when $\alpha>-1$ by various authors.  For the case $\alpha\in\mathbb{R}$ we refer to \cite{GKU2}.

The $\alpha\geq -1$ part of the the following pointwise estimate is proved in many places including \cite{CKY, JP,R}. For a proof  when $\alpha\in\mathbb{R}$ see \cite{GKU2}.
\begin{lemma}\label{Lemma-Kernel-Estimate}
Let $\alpha\in\mathbb{R}$ and $m$ be a multi-index. Then for every $x\in \mathbb{B}$, $y\in \overline{\mathbb{B}}$,
\begin{equation*}
\big|(\partial^m R_\alpha)(x,y)\big|
\lesssim\begin{cases}
1,&\text{if $\, \alpha+|m|<-n$};\\
1+\log \dfrac{1}{[x,y]},&\text{if $\, \alpha+|m|=-n$};\\
\dfrac{1}{[x,y]^{n+\alpha+|m|}},&\text{if $\, \alpha+|m|>-n$}.
\end{cases}
\end{equation*}
\end{lemma}

For any $s,t\in\mathbb{R}$, exactly the same upper bounds hold for $D^t_s R_\alpha(x,y)$, so we have the following estimate.

\begin{lemma}\label{Lemma-Estimate-Dst-Kernel}
Let $\alpha,s,t\in\mathbb{R}$ and $m$ be a multi-index. Then for every $x\in \mathbb{B}$, $y\in \overline{\mathbb{B}}$,
\begin{equation*}
\big|\partial^m (D^t_s R_\alpha)(x,y)\big|
\lesssim\begin{cases}
1,&\text{if $\, \alpha+t+|m|<-n$};\\
1+\log \dfrac{1}{[x,y]},&\text{if $\, \alpha+t+|m|=-n$};\\
\dfrac{1}{[x,y]^{n+\alpha+t+|m|}},&\text{if $\, \alpha+t+|m|>-n$}.
\end{cases}
\end{equation*}
\end{lemma}

Pointwise upper bounds lead to integral upper bounds.  The next lemma gives a certain estimate on weighted integrals of powers of $R_\alpha(x,y)$.
When $\alpha>-1$ and $c>0$, it is proved in \cite[Proposition 8]{M}. For the whole range $\alpha\in \mathbb{R}$ see \cite[Theorem 1.5]{GKU2}.

\begin{lemma}\label{norm-kernel}
Let $\alpha\in \mathbb{R}$, $0<p<\infty$ and $b>-1$. Set $c=p(n+\alpha)-(n+b)$. Then
\[
\int_{\mathbb B}|R_\alpha(x,y)|^p\,(1-|y|^2)^b\,d\nu(y)
\sim\begin{cases}
1,&\text{if $c<0$};\\
\noalign{\medskip}
1+\log\dfrac1{1-|x|^2},&\text{if $c=0$};\\
\noalign{\medskip}
\dfrac1{(1-|x|^2)^c},&\text{if $c>0$}.
\end{cases}
\]
\end{lemma}

Notice that the kernel $R_\alpha(x,y)$ is dominated by $1/[x,y]^{n+\alpha}$ by taking $|m|=0$ and $\alpha>-n$ in Lemma \ref{Lemma-Kernel-Estimate}. We will also need the following integral estimates of these dominating terms. For a proof see, for example, \cite[Proposition 2.2]{LS} or \cite[Lemma 4.4]{R}.

\begin{lemma}\label{Integral-[x,y]}
Let $b>-1$ and $s\in \mathbb{R}$. Then
\begin{equation*}
  \int_{\mathbb{B}} \frac{(1-|y|^2)^b}{[x,y]^{n+b+s}} \, d\nu(y) \sim
      \begin{cases}
         1, &\text{if $\, s<0$};\\
         1+\log \dfrac{1}{1-|x|^2}, &\text{if $\, s=0$}; \\
         \dfrac{1}{(1-|x|^2)^s}, &\text{if $\, s > 0$}.
     \end{cases}
\end{equation*}
\end{lemma}
Integral operators involving $R_\alpha(x,y)$ or the above dominating terms are widely used in the study of Bergman spaces. We mention one more integral estimate. For a proof see  \cite[Lemma 4.2]{R}.

\begin{lemma}\label{Lemma Estimate Integral wrt t}
Let $b>-1$, $c>0$ and $x,y\in \mathbb{B}$. Then
\begin{equation*}
  \int_0^1 \frac{(1-\tau)^{b}}{[\tau x,y]^{1+b+c}}\, d\tau \lesssim \frac{1}{[x,y]^c}.
\end{equation*}
\end{lemma}

\section{Subharmonic Behaviour and an Integral Inequality}\label{Section-subharmonic}
If $u$ is harmonic on a domain $\Omega \subset \mathbb{R}^{n}$, then $|u|^{p}$ is subharmonic on $\Omega$ when $1\leq p<\infty$. This is no longer true when $0<p<1$, nevertheless it is shown in \cite{Fefferman} and \cite{Kuran} that $|u|^{p}$ has subharmonic behaviour in the following sense:
There exists a constant $K \, (>1)$ depending only on $n$ and $p$ such that
\begin{equation}\label{equsubharmonic}
 |u(x)|^{p}\leq \frac{K}{r^{n}}\int_{B(x,r)} |u(y)|^{p} d\nu(y),
\end{equation}
whenever $B(x,r)=\{y:|y-x|<r\}\subset \Omega$. This immediately  leads to the following pointwise estimate.
\begin{lemma}\label{growth}
Let $0<p<\infty$ and $\alpha>-1$. Then
\[
|u(x)| \lesssim \frac{\|u\|_{b^p_\alpha}}{(1-|x|^2)^{(n+\alpha)/p}}
\]
for all $u\in b^p_\alpha$ and $x\in \mathbb{B}$.
\end{lemma}
\begin{proof}
Pick $x\in \mathbb{B}$ and let $r=(1-|x|)/2$. Applying (\ref{equsubharmonic}) (with $K=1$ when $p\geq 1$) and noting that $(1-|y|^2)\sim(1-|x|^2)$ when $y \in B(x,r)$, we obtain
\begin{align*}
|u(x)|^{p}&\leq \frac{K}{r^{n}}\int_{B(x,r)} |u(y)|^{p} d\nu(y)\sim \frac{K}{r^{n+\alpha}}\int_{B(x,r)} |u(y)|^{p} (1-|y|^2)^{\alpha} d\nu(y)\\
&\lesssim \frac{\| u \|^{p}_{b^{p}_\alpha}}{(1-|x|^2)^{n+\alpha}}.
\end{align*}
\end{proof}

The above estimate leads to the following integral inequality (See \cite[Lemma 5.1]{R}).
\begin{lemma}\label{Lemma-Special1}
Let $0<p<1$ and $\alpha>-1$. Set $\rho=(n+\alpha)/p-n$. Then
\begin{equation*}
 \int_{\mathbb{B}} |u(x)| (1-|x|^2)^\rho d\nu(x)\lesssim \| u \|_{b^{p}_\alpha}
\end{equation*}
for all $u \in b^{p}_\alpha$.
\end{lemma}
\begin{proof}
Write $|u(x)|=|u(x)|^{p}|u(x)|^{1-p}$ and estimate the second factor by using
Lemma \ref{growth}.
\end{proof}

Later we will need a generalization of Lemma \ref{Lemma-Special1} to products of two harmonic functions. Before this we first show that product of two harmonic functions have subharmonic behaviour.

\begin{lemma}\label{twoLemma-Special}
Let $\Omega \subset \mathbb{R}^{n}$ be a domain, $0< p<\infty$ and $u,v$ be harmonic on $\Omega$.
There exists a constant $K$ such that
\begin{equation}\label{twoequsubharmonic}
 |u(x)v(x)|^{p}\leq \frac{K}{r^{n}}\int_{B(x,r)} |u(y)v(y)|^{p} d\nu(y),
\end{equation}
for every $B(x,r)\subset \Omega$. Here the constant $K$ depends only on $n$ and $p$ and is independent of $x,r,u,v$.
\end{lemma}

\begin{proof}
This follows from a result of M. Pavlovi\'c. In Theorem 2 of \cite{Pavlovic} it is shown that if $f \in C^{1}(\Omega)$ and there exists a constant $C$ such that
\begin{equation}\label{twopav}
 |\nabla f(x)|\leq \frac{C}{r}\sup\{|f(y)|: y \in B(x,r)\} \quad \text{whenever} \quad B(x,r)\subset \Omega,
\end{equation}
then there exists a constant $K$ depending only on $C$, $n$ and $p$ such that
\begin{equation*}
 |f(x)|^{p}\leq \frac{K}{r^{n}}\int_{B(x,r)} |f(y)|^{p} d\nu(y) \quad \text{for every} \quad B(x,r)\subset \Omega.
\end{equation*}
So, it remains only to show that (\ref{twopav}) holds with $f=uv$.

It is easily verified by differentiating Poisson integral that if $u$ is harmonic on $B(x,r)$, then
\begin{equation*}
 |\nabla u(x)|\leq \frac{n}{r}\sup\{|u(y)|: y \in B(x,r)\}.
\end{equation*}
Since
\begin{equation*}
 \nabla (uv)=v\nabla u+u\nabla v,
\end{equation*}
it follows that
\begin{equation*}
 |\nabla (uv)(x)|\leq \frac{2n}{r}\sup\{|u(y)v(y)|: y \in B(x,r)\}
\end{equation*}
and this finishes the proof.
\end{proof}

We now generalize Lemma \ref{Lemma-Special1} to product of two harmonic functions.

\begin{lemma}\label{Lemma-Special}
Let $0<p<1$ and $\alpha>-1$. Set $\rho=(n+\alpha)/p-n$. Then
\begin{equation*}
 \int_{\mathbb{B}} |u(x)v(x)| (1-|x|^2)^\rho d\nu(x)\lesssim \| uv \|_{L^{p}_\alpha},
\end{equation*}
for every  $u,v \in h(\mathbb{B})$.
\end{lemma}
\begin{proof}
 Repeating the proof of Lemma \ref{growth} shows that the
pointwise estimate
\[
|u(x)v(x)| \lesssim \frac{\|uv\|_{L^p_\alpha}}{(1-|x|^2)^{(n+\alpha)/p}}
\]
holds by the subharmonic behaviour of $uv$. Then writing $|uv|=|uv|^{p}|uv|^{1-p}$ as in the proof of Lemma \ref{Lemma-Special1} and estimating the second factor we get the desired result.
\end{proof}

\section{Proof of Theorem \ref{Theorem-Equiv-Bloch}}\label{Section-Proof-of-T1}
In this section, we will prove Theorem \ref{Theorem-Equiv-Bloch}, first when $\alpha>-1$ and then we will deal with the general case $\alpha\in \mathbb{R}$.
As is mentioned before, the case $1\leq p<\infty$ and the  holomorphic analogues  of Theorem \ref{Theorem-Equiv-Bloch}
are already considered elsewhere. However the technical details are quite different so that we will not refer to other sources and give a complete proof to make this work self-contained .

For future reference we record the following simple lemma which is a special case of the reproducing formula (\ref{Reproducing}).
\begin{lemma}\label{Lemma-Reproducing-Special}
Let $0<p<1$, $\alpha>-1$, $\rho=(n+\alpha)/p-n$ and $s>\rho$. If $u \in b^{p}_\alpha$, then
\begin{equation*}
  u(x)=\int_{\mathbb{B}} R_s(x,y) u(y) d\nu_s(y) = \frac{1}{V_s}\int_{\mathbb{B}} R_s(x,y) u(y) (1-|y|^2)^s d\nu(y).
\end{equation*}
\end{lemma}

\begin{proof}
It is well known that the reproducing formula (\ref{Reproducing-Bergman}) is  true when $u\in b^1_\alpha$. On the other side, applying Lemma \ref{Lemma-Special1} shows  $u\in b^1_\rho\subseteq b^1_s$ and the lemma follows.
\end{proof}

We begin the proof of Theorem \ref{Theorem-Equiv-Bloch} with Bergman spaces and first order derivatives and build up
from there. The following result is standard and can be proved by more elementary techniques. We include a proof for completeness and to illustrate how it follows from the reproducing formula, Lemma \ref{Lemma-Special} and the kernel estimates. Later, we will employ this technique many times.

\begin{lemma}\label{Lemma N=1}
Let $0<p<1$, $\alpha >-1$ and $u\in h(\mathbb{B})$. The following are equivalent:
\begin{enumerate}
  \item[(a)] $u\in b^{p}_\alpha$.
  \item[(b)] $(1-|x|^2) |\nabla u (x)| \in L^p_\alpha$.
  \item[(c)] $(1-|x|^2) \mathcal{R} u (x) \in L^p_\alpha$.
\end{enumerate}
Moreover,
\begin{equation}\label{Equivalent Norms when N=1}
  \| u - u(0) \|_{b^{p}_\alpha} \sim \| (1-|x|^2) \, |\nabla u(x)| \, \|_{L^p_\alpha} \sim \| (1-|x|^2) \mathcal{R} u(x) \|_{L^p_\alpha}.
\end{equation}
\end{lemma}

\begin{proof}
(a) $\Rightarrow$ (b): Assume that $u\in b^{p}_\alpha$. Pick $s>(n+\alpha)/p-n$. By Lemma \ref{Lemma-Reproducing-Special},
\begin{equation*}
  u(x)-u(0)=\frac{1}{V_s} \int_{\mathbb{B}} R_s(x,y) \big(u(y)-u(0)\big) (1-|y|^2)^s d\nu(y).
\end{equation*}
Taking partial derivative and differentiating $R_s(x,y)$ in the first variable under the integral sign which is easily justified using (\ref{Rq-uniformly bounded}), we obtain
\begin{equation*}
  \frac{\partial u}{\partial x_i}(x)=\frac{1}{V_s} \int_{\mathbb{B}} \frac{\partial}{\partial x_i} R_s(x,y) \big(u(y)-u(0)\big) (1-|y|^2)^s d\nu(y).
\end{equation*}
We set $s=(n+\alpha')/p-n$. Note that $\alpha'>\alpha$. By Lemma \ref{Lemma-Special}
\begin{equation*}
  \left|\frac{\partial u}{\partial x_i}(x)\right|^{p} \lesssim  \int_{\mathbb{B}} \left|\frac{\partial}{\partial x_i} R_s(x,y)\right|^{p}\big|u(y)-u(0)\big|^{p} (1-|y|^2)^ {\alpha'} d\nu(y).
\end{equation*}
 Applying Lemma \ref{Lemma-Kernel-Estimate}, we get
\begin{equation*}
  (1-|x|^2)^{p} \left|\frac{\partial u}{\partial x_i}(x)\right|^{p} \lesssim (1-|x|^2)^{p} \int_{\mathbb{B}} \frac{1}{[x,y]^{p(n+s+1)}} \big|u(y)-u(0)\big|^{p} (1-|y|^2)^{\alpha'}  d\nu(y).
\end{equation*}
Integrating both sides against the measure $d\nu_{\alpha}$ and applying Fubini's theorem shows
\begin{equation*}
  \int_{\mathbb{B}}(1-|x|^2)^{p} \left|\frac{\partial u}{\partial x_i}(x)\right|^{p} d\nu_{\alpha}(x)\lesssim  \int_{\mathbb{B}}\big|u(y)-u(0)\big|^{p} \int_{\mathbb{B}} \frac{(1-|x|^2)^{p+\alpha}}{[x,y]^{p(n+s+1)}}  d\nu(x) d\nu_{ \alpha'}(y).
\end{equation*}
Estimating the inner integral above according to Lemma \ref{Integral-[x,y]}, we find that
\begin{equation*}
  \int_{\mathbb{B}}(1-|x|^2)^{p} \left|\frac{\partial u}{\partial x_i}(x)\right|^{p} d\nu_{\alpha}(x)\lesssim \int_{\mathbb{B}} \big|\big(u(y)-u(0)\big)\big|^{p}  d\nu_{\alpha}(y),
\end{equation*}
and this proves (a) implies (b).

(b) $\Rightarrow$ (c): This immediately follows from (\ref{Radial-Derivative}).

(c) $\Rightarrow$ (a): We first write $u$ in terms of $\mathcal{R}u $. This is done by simple calculus and (\ref{Radial-Derivative}):
\begin{equation}\label{simple calc}
  u(x)-u(0)=\int_{0}^{1} \frac{\mathcal{R}u(\tau x)}{\tau}d\tau.
\end{equation}

Pick $s$ such that $s>(n+\alpha+p)/p-n$. Notice that $s>0$, so $n+s-1>0$. We are given that $\mathcal{R} u \in b^{p}_{\alpha+p}$, so by Lemma \ref{Lemma-Reproducing-Special}
\begin{equation*}
  \mathcal{R}u(x)=\int_{\mathbb{B}} R_s(x,y) \mathcal{R} u(y) d\nu_s(y).
\end{equation*}
By setting $x=0$ and using (\ref{Radial-Derivative}) and (\ref{Rq(x,0)}), we obtain $0=\int_{\mathbb{B}}\mathcal{R} u(y) d\nu_s(y)$. Subtracting this
from the previous equation yields
\begin{equation*}
  \mathcal{R}u(x)=\int_{\mathbb{B}} \big(R_s(x,y)-1 \big) \mathcal{R} u(y) d\nu_s(y).
\end{equation*}
Using (\ref{simple calc}) and the Fubini's theorem, we obtain
\begin{align*}
 u(x)-u(0)&= \int_{0}^{1}\frac{1}{\tau}\int_{\mathbb{B}} \big(R_s(\tau x,y)-1 \big) \mathcal{R} u(y) d\nu_s(y) d\tau \\
          &= \int_{\mathbb{B}} \mathcal{R} u(y)\int_{0}^{1}\frac{R_s(\tau x,y)-1 }{\tau}d\tau d\nu_s(y).
\end{align*}

Let
\begin{equation*}
 G(x,y):=\int_{0}^{1}\frac{R_s(\tau x,y)-1 }{\tau}d\tau
\end{equation*}
which is harmonic in $y$. Set $s=(n+\alpha')/p-n$. Notice that $\alpha'>\alpha+p$.  By Lemma \ref{Lemma-Special}
\begin{equation}\label{est-G}
 |u(x)-u(0)|^{p}\lesssim \int_{\mathbb{B}} |\mathcal{R} u(y)|^{p}|G(x,y)|^{p}d\nu_ {\alpha'}(y).
\end{equation}
To estimate $|G(x,y)|$, we first show that
\begin{equation}\label{for-est-G}
 \frac{|R_s(\tau x,y)-1|}{\tau} \lesssim \frac{1}{[\tau x,y]^{n+s}}, \quad 0\leq\tau<1.
\end{equation}
If $1/2<\tau<1$, Lemma \ref{Lemma-Kernel-Estimate} implies
$|R_s(\tau x,y)-1|/\tau\leq 2|R_s(\tau x,y)-1|\lesssim 1/[\tau x,y]^{n+s}$. If $0\leq \tau \leq 1/2$, by (\ref{Rq - Series expansion1}) $\big(R_s(\tau x,y)-1 \big)/\tau =\sum_{k=1}^\infty \gamma_k(s)\tau^{k-1}  Z_k(x,y)$ and the series  uniformly converges in $x,y,\tau$ by Lemma \ref{Lemma-Zk} (c) and (\ref{gamma_k}) and hence is uniformly bounded. Thus (\ref{for-est-G}) holds and applying   Lemma \ref{Lemma Estimate Integral wrt t} shows
\begin{equation*}
 |G(x,y)|\leq \int_{0}^{1}\frac{|R_s(\tau x,y)-1| }{\tau}d\tau\lesssim \int_{0}^{1}\frac{d\tau}{[\tau x,y]^{n+s}}\lesssim \frac{1}{[ x,y]^{n+s-1}}.
\end{equation*}
Inserting this into (\ref{est-G}) we obtain
\begin{equation*}
 |u(x)-u(0)|^{p}\lesssim \int_{\mathbb{B}} \frac{|\mathcal{R} u(y)|^{p}}{[x,y]^{p(n+s-1)}}d\nu_ {\alpha'}(y).
\end{equation*}
Finally, Fubini's theorem and Lemma \ref{Integral-[x,y]} yields

\begin{align*}
 \int_{\mathbb{B}}  |u(x)-u(0)|^{p}d\nu_\alpha(x)&\lesssim \int_{\mathbb{B}}|\mathcal{R} u(y)|^{p}\int_{\mathbb{B}}\frac{ d\nu_ {\alpha}(y)}{[x,y]^{p(n+s-1)}}d\nu_ {\alpha'}(y)\\
 &\lesssim \int_{\mathbb{B}} \left( (1-|y|^{2})\big|\mathcal{R} u(y)\big|\right)^{p} d\nu_ {\alpha}(y).
\end{align*}
This completes the proof.
\end{proof}

We note that (\ref{Equivalent Norms when N=1}) also  can  be written  in the following form:
\begin{equation*}
   \| u \|_{b^{p}_\alpha} \sim |u(0)|+ \| (1-|x|^2) \, |\nabla u(x)| \, \|_{L^p_\alpha} \sim |u(0)| + \| (1-|x|^2) \mathcal{R} u(x) \|_{L^p_\alpha}.
\end{equation*}

It is straightforward to extend the previous lemma to higher order derivatives.

\begin{lemma}\label{Lemma N bigger than 1}
Let $0<p<1$ and $\alpha>-1$. Then the following conditions are equivalent for $u\in h(\mathbb{B})$:
\begin{enumerate}
\item[(a)] $u\in b^{p}_\alpha$.
\item[(b)] For every $N\in \mathbb{N}$, we have $(1-|x|^2)^N \partial^m u \in L^p_\alpha$ for every multi-index $m$ with $|m|=N$.
\item[(c)] There exists $N\in\mathbb{N}$ such that $(1-|x|^2)^N \partial^m u \in L^p_\alpha$ for every multi-index $m$ with $|m|=N$.
\item[(d)] For every $N\in \mathbb{N}$, we have $(1-|x|^2)^N \mathcal{R}^N u \in L^p_\alpha$.
\item[(e)] There exists $N\in\mathbb{N}$ such that $(1-|x|^2)^N \mathcal{R}^{N} u \in L^p_\alpha$.
\end{enumerate}
Moreover,
\begin{equation}\label{Equivalent-norms-N-bigger-than1}
\begin{split}
\| u \|_{b^{p}_\alpha} & \sim \sum_{|m|\leq N-1} |(\partial^m u) (0)| + \sum_{|m|=N} \| (1-|x|^2)^N \partial^m u(x) \|_{L^p_\alpha} \\
                     & \sim |u(0)| + \|(1-|x|^2)^N \mathcal{R}^N u(x) \|_{L^p_\alpha}.
\end{split}
\end{equation}
\end{lemma}

\begin{proof}
We will show that (a) $\Leftrightarrow$ (b) $\Leftrightarrow$ (c). The equivalence (a) $\Leftrightarrow$ (d) $\Leftrightarrow$ (e) can be justified similarly.
First, there is nothing to prove for (b) $\Rightarrow$ (c).

(a) $\Rightarrow$ (b): Assume that $u\in b^{p}_\alpha$. By Lemma \ref{Lemma N=1}, $\dfrac{\partial u}{\partial x_i} \in b^{p}_{\alpha+p}$ for e\-very $i=1,2,\ldots,n$. If we apply Lemma \ref{Lemma N=1} again we obtain $\dfrac{\partial^2 u}{\partial x_j \partial x_i} \in b^{p}_{\alpha+2p}$ for e\-very $i,j=1,2,\ldots,n$. We continue so on until we obtain $\partial^m u \in b^{p}_{\alpha+pN}$ for every $m$ with $|m|=N$.

(c) $\Rightarrow$ (a): Assume that $(1-|x|^2)^N \partial^m u \in L^p_\alpha$, that is $\partial^m u \in b^{p}_{\alpha+pN}$ for every multi-index $m$ with $|m|=N$. Let $m'$ be a multi-index with $|m'|=N-1$. Then $\dfrac{\partial}{\partial x_i} \partial^{m'}u\in b^{p}_{\alpha+pN}$ for every $i=1,2,\ldots,n$ and Lemma \ref{Lemma N=1} implies that $\partial^{m'} u \in b^{p}_{\alpha+p(N-1)}$. We repeat the same argument sufficiently many times until  $u\in b^{p}_\alpha$ is obtained.

It is easy to verify (\ref{Equivalent-norms-N-bigger-than1}) and  we omit the details.
\end{proof}

We now prove counterpart of Lemma \ref{Lemma N=1} for the operators $D^t_s$. We continue to stay in the region $\alpha>-1$.

\begin{lemma}\label{Lemma Dst}
Let $0<p<1$ and $\alpha>-1$. Then the following conditions are equivalent for $u\in h(\mathbb{B})$:
\begin{enumerate}
\item[(a)] $u \in b^{p}_\alpha$.
\item[(b)] For every $s,t\in \mathbb{R}$ with $\alpha+pt>-1$, we have $(1-|x|^2)^t D^t_s u \in L^p_\alpha$.
\item[(c)] There exist $s,t\in \mathbb{R}$ with $\alpha+pt>-1$ such that $(1-|x|^2)^t D^t_s u \in L^p_\alpha$.
\end{enumerate}
Moreover, $\|u\|_{b_\alpha} \sim \|(1-|x|^2)^t D^t_s u\|_{L^p_\alpha}$.
\end{lemma}

\begin{proof}
With (b) $\Rightarrow$ (c) being clear, it is enough to show that (a) $\Rightarrow$ (b) and (c) $\Rightarrow$ (a).

(a) $\Rightarrow$ (b): Suppose $u\in b^{p}_\alpha$. Pick $c>(n+\alpha)/p-n$. Then by Lemma \ref{Lemma-Reproducing-Special},
\begin{equation*}
  u(x)=\int_{\mathbb{B}} R_c(x,y) u(y) d\nu_c(y).
\end{equation*}
If we apply $D^t_s$ to both sides and push it into the integral using Lemma \ref{Lemma-Push-Dst}, we obtain
 \begin{equation*}
  (1-|x|^2)^t |D^t_s u(x)| \lesssim (1-|x|^2)^t \int_{\mathbb{B}} |D^t_s R_c(x,y)| |u(y)| (1-|y|^2)^c d\nu(y).
\end{equation*}
We set $c=(n+\alpha')/p-n$. Note that  $\alpha'>\alpha$. Lemma \ref{Lemma-Special} shows
\begin{equation*}
(1-|x|^2)^{pt} |D^t_s u(x)|^{p}   \lesssim (1-|x|^2)^{pt} \int_{\mathbb{B}} \left|D^t_s R_c(x,y)\right|^{p}\big|u(y)\big|^{p} (1-|y|^2)^ {\alpha'} d\nu(y).
\end{equation*}
Applying Lemma \ref{Lemma-Estimate-Dst-Kernel} (with $n+c+t>(n+\alpha)/p+t>(n-1)/p>0$), we get
\begin{equation*}
  (1-|x|^2)^{pt}|D^t_s u(x)|^{p} \lesssim (1-|x|^2)^{pt} \int_{\mathbb{B}}\frac{1}{[x,y]^{p(n+c+t)}} \big|u(y)\big|^{p} (1-|y|^2)^ {\alpha'} d\nu(y).
\end{equation*}

An application of  Fubini's theorem in combination with Lemma \ref{Integral-[x,y]} shows that
\begin{equation*}
  \int_{\mathbb{B}} (1-|x|^2)^{pt} |D^t_s u(x)|^{p} d\nu_{\alpha}(x)\lesssim \int_{\mathbb{B}} |u(y)|^{p} d\nu_{\alpha}(y),
\end{equation*}
and part (b) follows.

(c) $\Rightarrow$ (a): The condition $(1-|x|^2)^t D^t_s u \in L^p_\alpha$ implies $D^t_s u\in L^{p}_{\alpha+pt}$ and since $D^t_s u$ is harmonic and $\alpha+pt>-1$, $D^t_s u\in b^{p}_{\alpha+pt}$. The part (a) $\Rightarrow$ (b) shows
 $(1-|x|^{2})^{-t} D^{-t}_{s+t}D^t_s u \in L^{p}_{\alpha+pt}$. By (\ref{inverse of Dst}), $(1-|x|^{2})^{-t}  u \in L^{p}_{\alpha+pt}$ and hence
 $u \in L^{p}_{\alpha}$. As $u$ is harmonic and $\alpha>-1$ we conclude that $u\in b^{p}_{\alpha}$.
\end{proof}

Before removing the restriction $\alpha>-1$ and extending the previous lemmas to all $\alpha\in \mathbb{R}$ we mention the following two lemmas. These are elementary but we include proofs for completeness.

\begin{lemma}\label{Radial-in-terms-of-partials}
Let $N\geq 1$ be an integer. Then
\begin{equation*}
\mathcal{R}^N = \sum_{1\leq |m|\leq N} p_m \partial^m,
\end{equation*}
where $p_m$ is a polynomial with degree equal to $|m|$.
\end{lemma}
\begin{proof}
Let $u$ be a smooth function. For $N=1$, $\mathcal{R} u(x)= x\cdot \nabla u(x)=\sum_{i=1}^n x_i  \partial u/\partial x_i$, so the lemma is true in this case. For $N=2$ we compute
\begin{equation*}
 \mathcal{R}^2 u(x) = \sum_{j=1}^n x_j \frac{\partial}{\partial x_j} \left(\sum_{i=1}^n x_i \frac{\partial u}{\partial x_i} \right) = \sum_{i,j=1}^n x_i x_j \frac{\partial^2 u}{\partial x_j \partial x_i} + \sum_{j=1}^n x_j \frac{\partial u}{\partial x_j}
\end{equation*}
and the lemma is also true for $N=2$. The general case follows from induction.
\end{proof}

\begin{lemma}\label{Radial-in-terms-of-partials low}
Let $\alpha>-1$, $0<p<1$ and $u\in h(\mathbb{B})$. If $\partial^m u \in b^{p}_{\alpha}$ for every multi-index $m$ with $|m|=N$, then $\partial^{m'} u\in b^{p}_{\alpha}$ for all multi-indices $m'$ with $|m'|<N$.
\end{lemma}
\begin{proof}
 Suppose that $\partial^m u \in b^{p}_{\alpha}$ for every multi-index $m$ with $|m|=N$ and $|m'|<N$. Then by Lemma \ref{Lemma N bigger than 1}, $\partial^{m'}\partial^m u \in b^{p}_{\alpha+p|m'|}$ and by (\ref{Inclusion}), $\partial^m \partial^{m'} u = \partial^{m'} \partial^m u \in b^{p}_{\alpha+pN}$. Since this is true for every $|m|=N$, we conclude by Lemma \ref{Lemma N bigger than 1} that $\partial^{m'} u\in b^{p}_{\alpha}$.
\end{proof}

With all the preparations above, now we can give the proof of Theorem \ref{Theorem-Equiv-Bloch} about the interchangeability of various
kinds of derivatives in defining harmonic Besov spaces $b^{p}_{\alpha}$, $0<p<1$.

\begin{proof}[Proof of Theorem \ref{Theorem-Equiv-Bloch}]
Since the implications (a) $\Rightarrow$ (b), (c) $\Rightarrow$ (d) and (e) $\Rightarrow$ (f) are clear,
we only need to prove implications (b) $\Rightarrow$ (c), (d) $\Rightarrow$ (e) and (f) $\Rightarrow$ (a).
 We will refer many times to Lemmas \ref{Lemma N bigger than 1} and \ref{Lemma Dst} and will make sure that the subscript of $b^{p}$ is always greater than $-1$ in these cases.

(b) $\Rightarrow$ (c): Assume there exists $N_0$ with $\alpha+pN_0>-1$ such that $(1-|x|^2)^{N_0} \partial^m u \in L^p_\alpha$, that is $\partial^m u \in b^{p}_{\alpha+pN_0}$ for every multi-index $m$ with $|m|=N_0$. Then by Lemma \ref{Radial-in-terms-of-partials low}, $\partial^{m'} u\in b^{p}_{\alpha+pN_0}$ for all $m'$ with $|m'|<N_0$. Applying Lemma \ref{Radial-in-terms-of-partials} shows that
\begin{equation}\label{Radial N0}
\mathcal{R}^{N_0} u \in b_{\alpha+pN_0}.
\end{equation}

Now take any $N\in\mathbb{N}$ such that $\alpha+pN>-1$. If $N>N_0$, Lemma \ref{Lemma N bigger than 1} and (\ref{Radial N0}) implies $\mathcal{R}^N u = \mathcal{R}^{N-N_0} (\mathcal{R}^{N_0} u) \in b^{p}_{\alpha+pN_0+p(N-N_0)}=b_{\alpha+pN}$. If $N<N_0$, writing $\mathcal{R}^{N_0} u = \mathcal{R}^{N_0-N} (\mathcal{R}^N u)$ and applying again Lemma \ref{Lemma N bigger than 1} and (\ref{Radial N0}) implies $\mathcal{R}^N u \in b_{\alpha+pN_0-p(N_0-N)}=b_{\alpha+pN}$.

(d) $\Rightarrow$ (e): Assume there exists $N_0\in \mathbb{N}$ with $\alpha+pN_0>-1$ such that $(1-|x|^2)^{N_0} \mathcal{R}^{N_0} u \in L^p_\alpha$, that is $\mathcal{R}^{N_0} u \in b^{p}_{\alpha+pN_0}$. Take any $s,t \in \mathbb{R}$ such that $\alpha+pt>-1$. By Lemma \ref{Lemma Dst}, we have $D^t_s (\mathcal{R}^{N_0} u) \in b^{p}_{\alpha+pN_0+pt}$. By considering their actions on homogeneous expansions it is clear that the operators $D^t_s$ and $\mathcal{R}^{N_0}$ commute. Therefore $\mathcal{R}^{N_0} (D^t_s u) \in b^{p}_{\alpha+pN_0+pt}$. Hence we conclude that $D^t_s u \in b^{p}_{\alpha+pt}$ by Lemma \ref{Lemma N bigger than 1}.

(f) $\Rightarrow$ (a): Assume there exist $s_0, t_0 \in \mathbb{R}$ with $\alpha+pt_0>-1$ such that $(1-|x|^2)^{t_0} D^{t_0}_{s_0} u\in L^p_\alpha$, that is $D^{t_0}_{s_0}u \in b^{p}_{\alpha+pt_0}$. Pick $c$ with $c>(n+\alpha+pt_{0})/p-n$. Then by Lemma \ref{Lemma-Reproducing-Special},
\begin{equation*}
  D^{t_0}_{s_0} u(x) = \int_{\mathbb{B}} R_c(x,y) D^{t_0}_{s_0} u(y) d\nu_c(y).
\end{equation*}
We apply $D^{-t_0}_{s_0+t_0}$ to both sides, use (\ref{inverse of Dst}) on the left, push $D^{-t_0}_{s_0+t_0}$ into the integral by Lemma \ref{Lemma-Push-Dst} to obtain
\begin{equation*}
  u(x)=\int_{\mathbb{B}} D^{-t_0}_{s_0+t_0} R_c(x,y) \, D^{t_0}_{s_0} u(y) d\nu_c(y).
\end{equation*}
Take $N\in\mathbb{N}$ with $\alpha+pN>-1$ and let $m$ be a multi-index with $|m|=N$. Differentiating under the integral sign, we obtain
\begin{align*}
  \partial^m u (x)= \int_{\mathbb{B}} \partial^m \left(D^{-t_0}_{s_0+t_0} R_c(x,y)\right) D^{t_0}_{s_0} u(y) d\nu_c(y).
\end{align*}
Set $c=(n+\alpha')/p-n$. Note that $\alpha'>\alpha+pt_{0}$. Apply Lemma \ref{Lemma-Special} to obtain
\begin{equation*}
  (1-|x|^2)^{pN} |\partial^m u(x)|^{p} \lesssim (1-|x|^2)^{pN} \int_{\mathbb{B}} \big| \partial^m \left(D^{-t_0}_{s_0+t_0} R_c(x,y)\right)\big|^{p} \big|D^{t_0}_{s_0}u(y)\big|^{p} d\nu_{\alpha'}(y).
\end{equation*}
Applying Lemma \ref{Lemma-Estimate-Dst-Kernel} (with $n+c-t_0+N>(n+\alpha)/p+N-1>(n-1)/p>0$), we get
\begin{equation*}
  (1-|x|^2)^{pN} |\partial^m u(x)|^{p} \lesssim (1-|x|^2)^{pN} \int_{\mathbb{B}}  \frac{1}{[x,y]^{p(n+c-t_0+N)}} \big|D^{t_0}_{s_0}u(y)\big|^{p} d\nu_{\alpha'}(y).
\end{equation*}
Finally, using Fubini's theorem and Lemma \ref{Integral-[x,y]} we deduce that
\begin{equation*}
  \int_{\mathbb{B}}(1-|x|^2)^{pN} |\partial^m u(x)|^{p} d\nu_{\alpha}(x)\lesssim \int_{\mathbb{B}} (1-|y|^{2})^{pt_{0}} |D^{t_0}_{s_0}u(y)|^{p} d\nu_{\alpha}(y).
\end{equation*}

By retracing the above proof it is not hard to see that (\ref{Norm}) holds. We omit the details.
\end{proof}

We now verified all the requirements of Definition \ref{DefinitionAllAlpha} and  have the two-parameter Besov space family $b^{p}_\alpha$ with $0<p<1$ and $\alpha\in\mathbb{R}$ at hand . We will now show the basic properties of these spaces that can be obtained directly from this definition.

 As mentioned in the beginning, by (\ref{Norm}) we can endow $b^{p}_\alpha$, $0<p<1$  with many equivalent quasinorms. From now on
we will mainly use the quasinorms induced by $D^t_s$: Given $\alpha\in \mathbb{R}$, pick any $s,t$ with $\alpha+pt>-1$, then $\|(1-|x|^2)^t D^t_s u\|_{L^p_\alpha} = \| I^t_s u\|_{L^p_\alpha}$ is a quasinorm on $b^{p}_\alpha$; all these quasinorms are equivalent and we will denote any one of them by $\| \cdot \|_{b^{p}_\alpha}$ without indicating the dependence on $s$ and $t$.

Our first result shows that for fixed $p$, the spaces  $b^{p}_\alpha$  are all isomorphic. We will use later this important property many times, since it allows us to pass from one Besov space to another. We emphasize that the proposition below is true for every $t\in \mathbb{R}$ without any restriction.

\begin{proposition} \label{Proposition-The map Dst}
Let $0<p<1$ and $\alpha\in \mathbb{R}$. For any $s,t\in\mathbb{R}$, the map $D^t_s : b^{p}_\alpha \to b^{p}_{\alpha+pt}$ is an isomorphism and is an isometry when appropriate quasinorms are used in the two spaces.
\end{proposition}

\begin{proof}
Let $u \in h(\mathbb{B})$ and put $v=D^t_s u \in h(\mathbb{B})$. Pick $t_1$ such that $\alpha+p(t+t_1)>-1$. Then by (\ref{Additive-Dst}), $D^{t_{1}}_{s+t} v =D^{t_{1}}_{s+t}D^t_s u = D^{t+t_{1}}_{s}u$. If $u \in b^{p}_\alpha$, then $I^{t+t_{1}}_{s}u \in L^p_{\alpha}$, and this is equivalent to $I^{t_{1}}_{s+t} v\in L^p_{\alpha+pt}$ which means by Definition \ref{DefinitionAllAlpha}, $v \in b^{p}_{\alpha+pt}$. On the other hand, if $v \in b^{p}_{\alpha+pt}$, then $I^{t_{1}}_{s+t} v\in L^p_{\alpha+pt}$ and this is equivalent to $I^{t+t_{1}}_{s}u \in L^p_{\alpha}$. This means  by Definition \ref{DefinitionAllAlpha} again $u \in b^{p}_{\alpha}$. Together with (\ref{inverse of Dst}), the isomorphism claim follows.

 We endow $b^{p}_\alpha$ with the quasinorm $\|u\|_{b^{p}_\alpha}=\| I^{t+t_{1}}_s u\|_{L^p_\alpha}$ and $b^{p}_{\alpha+pt}$ with the quasinorm $\|v\|_{b^{p}_{\alpha+pt}} = \|I^{t_{1}}_{s+t} \|_{L^p_{\alpha+pt}}$. By (\ref{Additive-Dst}),
\begin{align*}
  \|v\|_{b^{p}_{\alpha+pt}} &= \| I^{t_{1}}_{s+t} D^t_s u\|_{L^p_{\alpha+pt}} = \|(1-|x|^2)^{t_{1}} D^{t_{1}}_{s+t} (D^t_s u) \|_{L^p_{\alpha+pt}} \\
                             &= \|(1-|x|^2)^{t_{1}} D^{t_{1}+t}_{s} u\|_{L^p_{\alpha+pt}}=\|I^{t_{1}+t}_{s}u \|_{L^p_{\alpha}}=\|u\|_{b^{p}_\alpha}. \qedhere
\end{align*}
and this proves the claim on isometry.
\end{proof}
For a proof of the above lemma when $1\leq p<\infty$ see \cite[Corollary 9.2]{GKU2}.
Let $0 < r <1$ and $u_r : \mathbb{B}\to \mathbb{C}$, $u_r(x)=u(rx)$ be the dilate of $u$.

\begin{corollary}\label{Corollary-Basic Properties}
Let $0<p<1$ and $\alpha\in \mathbb{R}$. The following properties hold:
\begin{enumerate}
\item[(a)] The space $b^{p}_\alpha$ is a quasi-Banach space; i.e. it is a complete metric space with metric
\begin{equation*}
d(u,v)=\|u-v\|_{b^{p}_\alpha}^{p}
\end{equation*}
which satisfies the properties $d(u,v)=d(u-v,0)$ and $d(\lambda u, 0)=|\lambda|^{p}d(u,0)$ for $\lambda\in \mathbb{C}$.
\item[(b)]The set of harmonic polynomials is dense in each space $b^{p}_\alpha$. In particular, $u_r \to u$ (as $r\to 1^-$) in $b^{p}_\alpha$.
\item[(c)] $b^{p}_{\alpha}$ is separable.
\end{enumerate}
\end{corollary}

\begin{proof}
  It is well known that these properties hold for the unweighted spaces $b^{p}_0$. It is also elementary to verify them for the weighted ones
$b^{p}_{\alpha}$, $\alpha>-1$, too. The $b^{p}_{\alpha}$ for $\alpha\leq-1$ then follows from the isomorphism in Proposition \ref{Proposition-The map Dst}, the fact that $D^t_s$ maps polynomials to polynomials and the simple identity $D^t_s (u_r)= (D^t_s u)_r$.
\end{proof}

In particular, the spaces $b^{p}_{\alpha}$ are F-spaces. One of the properties of an F-space
is that the closed graph theorem is valid for it.

We next prove the following embedding of harmonic Bergman spaces. In the special case $\alpha>-1$, it is just Lemma \ref{Lemma-Special1}.
For a complete description of all possible  inclusion relations between harmonic Besov spaces, see  \cite{DU2}.

\begin{proposition}\label{embedding-inc}
Let $0<p<1$, $\alpha\in \mathbb{R}$. If $\rho=(n+\alpha)/p-n$, Then
$ b^{p}_\alpha \subset b^{1}_{\rho} $ and the inclusion is continuous.
\end{proposition}

\begin{proof}
Pick $s,t \in \mathbb{R}$ and $t$ large enough such that $\alpha+pt>-1$ and $t+\rho>-1$. Let $u \in b^{p}_\alpha $. Then by Proposition \ref{Proposition-The map Dst} it follows that $D^t_s u \in  b^{p}_{\alpha+pt} $. By Lemma \ref{Lemma-Special1}, we have $D^t_s u \in  b^{1}_{t+\rho}$, where $(n+\alpha+pt)/p-n=t+\rho$. Equivalently, the function $(1-|x|^{2})^{t}D^t_s u$ belongs to $ L^{1}_{\rho}$, that is $ u  \in b^{1}_{\rho}$.
\end{proof}

\section{Pointwise Estimates, Duality and Atomic Decomposition}\label{pe-dual-atomic}

\subsection{Pointwise Estimates}\label{subsection-Pointwise Estimates}
We often need to know how fast a function in $b^{p}_{\alpha}$ grows near the boundary of $\mathbb{B}$. In this subsection, using Proposition \ref{embedding-inc}
and Theorem 13.1 of \cite{GKU2}, we obtain uniform growth rates for all $u \in b^{p}_{\alpha}$ near the boundary of $\mathbb{B}$. More precisely, we have the following result.
\begin{theorem}\label{Pestimates}
Let $0<p<1$ and $\alpha\in \mathbb{R}$. Then for all $u \in b^{p}_{\alpha}$ and $x\in \mathbb{B}$,
\begin{equation*}
  |u(x)| \lesssim \|u\|_{b^{p}_{\alpha}}
      \begin{cases}
       1, &\text{if $\, \alpha\leq -n$};\\
         \dfrac{1}{(1-|x|^2)^{(n+\alpha)/p}}, &\text{if $\, \alpha > -n$}.
        \end{cases}
\end{equation*}
\end{theorem}

\begin{proof}
Suppose $u \in b^{p}_{\alpha}$. Then by Proposition \ref{embedding-inc} $ u \in  b^{1}_{\rho}$ where $\rho=(n+\alpha)/p-n$. So if we apply
Theorem 13.1 of \cite{GKU2} for $p=1$, and $q=\rho$ with $\|u\|_{b^{1}_{\rho}} \lesssim\|u\|_{b^{p}_{\alpha}}$, we obtain all the growth rates at once.
\end{proof}

As a corollary of the above theorem, point evaluations are bounded linear functionals on $b^{p}_{\alpha}$. Holomorphic analogues of this theorem can be found in \cite[Section 6]{ZZ}.

Derivatives of functions in $ b^{p}_{\alpha}$ also have  uniform growth rates near the boundary of $\mathbb{B}$.
\begin{corollary}
Let $0<p<1$ and $\alpha\in \mathbb{R}$. Then for any $s,t\in \mathbb{R}$
\begin{equation*}
  |D^t_s u(x)| \lesssim \|u\|_{b^{p}_{\alpha}}
      \begin{cases}
       1, &\text{if $\, \alpha+pt\leq-n$};\\
         \dfrac{1}{(1-|x|^2)^{(n+\alpha+pt)/p}}, &\text{if $\, \alpha+pt > -n$}.
        \end{cases}
\end{equation*}
\end{corollary}

\begin{proof}
It suffices to apply first Proposition \ref{Proposition-The map Dst} and then Theorem \ref{Pestimates}.
\end{proof}
We can prove a little more than Theorem \ref{Pestimates} by using polynomial approximations.
\begin{proposition}\label{Pestimates1}
Let $0<p<1$ and $\alpha\in \mathbb{R}$. If $\alpha>-n$ and $u \in b^{p}_{\alpha}$, then
\begin{equation*}
 \lim_{|x|\to 1^{-}} (1-|x|^{2})^{(n+\alpha)/p}\big|u(x)\big| =0
\end{equation*}
\end{proposition}
\begin{proof}
Let $u_r(x)=u(rx)$, $0 < r <1$  be the dilate of $u$. Apply  $\alpha>-n$ part of  Theorem \ref{Pestimates} to $u-u_{r}$ and multiply both sides by $(1-|x|^{2})^{(n+\alpha)/p}$. Let $\varepsilon>0$. By Corollary \ref{Corollary-Basic Properties} (b), there is a $0<r_{0}<1$ such that
$(1-|x|^{2})^{(n+\alpha)/p}\big|u(x)-u_{r_{0}}(x)\big|<\varepsilon$ for all $x\in \mathbb{B}$. Then $(1-|x|^{2})^{(n+\alpha)/p}\big|u(x)\big|<\varepsilon+(1-|x|^{2})^{(n+\alpha)/p}\big|u_{r_{0}}(x)\big|$ for all $x\in \mathbb{B}$.  Letting $|x|\to 1^{-}$ and noting that $u_{r_{0}}$ is bounded in $\overline{\mathbb{B}}$, we obtain the desired result.
\end{proof}
\begin{proposition}\label{Pestimates2}
Let $0<p<1$ and $\alpha\in \mathbb{R}$. Suppose $\alpha\leq-n$, then every function in $ b^{p}_{\alpha}$ is continuous on $\overline{\mathbb{B}}$.
\end{proposition}
\begin{proof}
Let $u\in b^{p}_{\alpha}$ with $\alpha\leq-n$. Then by Proposition \ref{embedding-inc} $ u \in  b^{1}_{\rho}$ where $\rho=(n+\alpha)/p-n$. As $\rho \leq -n$, the required result follows from Theorem 13.2 of \cite{GKU2}.
\end{proof}

\subsection{Duality}\label{subsection-Duality}
In this subsection we will prove Theorem \ref{Theorem-Dual-of-b1q}. More precisely, we will show that $(b^p_\alpha)'$, $0<p<1$ can be identified with $b^{\infty}_\beta$ for any $\alpha,\beta \in \mathbb{R}$ without any restriction. The aforementioned identification can be obtained using many different pairings. By results of the previous section, the point evaluation at any $x \in \mathbb{B}$ is
a bounded linear functional on $b^p_\alpha$. Therefore, $(b^p_\alpha)'$ is a nontrivial Banach space for all  $0<p<1$ and all real $\alpha$.

At the beginning, we defined harmonic Bloch spaces $b^\infty_\alpha$ in terms of partial derivatives. Analogous to Theorem \ref{Theorem-Equiv-Bloch}  we have the following: Given $\alpha\in\mathbb{R}$, pick $s,t\in\mathbb{R}$ such that $\alpha+t>0$. Then $u\in h(\mathbb{B})$ belongs to $b^\infty_\alpha$ if and only if $I^t_s u \in L^\infty_\alpha$ and $\| I^t_s u \|_{L^\infty_\alpha}$ is a norm on $b^\infty_\alpha$ (see Theorem 1.2 of \cite{DU1}).

When $0<p<1$, comparing with the arguments in the proves of \cite[Theorem 13.4]{GKU2} and  \cite[Theorem 5.4]{DU1} for the case $1\leq p<\infty$, we are led to deal with some obstructions. First, $b^p_\alpha$  is not locally convex for $0<p<1$, so the Hahn-Banach theorem fails to hold. If we assume that it holds, $(L^p_\alpha)'= {0}$. This clearly cannot be the way to proceed, as $(b^p_\alpha)'$ contain the point-evaluations. Another obstruction is that we can not express $u\in b^p_\alpha$ with the help of integral representation ensuing from Bergman-Besov projections. From this point, the embedding result Proposition \ref{embedding-inc} will be crucial in the proof of theorem and it allows us to use integral representation for the function in  $b^p_\alpha$.

Note that our proof also suits to the case $p = 1$. On the other hand, Theorem 5.4 of \cite{DU1} and Teorem \ref{Theorem-Dual-of-b1q}  show that the dual space of $b^p_\alpha$ is isomorphic to that of $b^{1}_{\rho}$ for $0<p<1$, where $\rho=(n+\alpha)/p-n$. We are now ready to prove the theorem.

\begin{proof}[Proof of Theorem \ref{Theorem-Dual-of-b1q}]
 Let $t'=s-\rho-\beta$ and $s'=t+\rho+\beta$. Then by (\ref{Dual-S}) and (\ref{Dual-T}), we have
\begin{align}
  s' &> \beta-1, \label{Dual-S'}\\
  \beta +t' &> 0. \label{Dual-T'}
\end{align}
If $v\in b^{\infty}_\beta$, then $I^{t'}_{s'} v \in L^\infty_\beta$ by (\ref{Dual-T'}) and if $u\in b^p_\alpha$, then $I^t_s u\in L^p_\alpha$ by (\ref{Dual-T}). Therefore the pairing (\ref{Dual-Pairing}) defines a bounded linear functional namely $\mathcal{L}_{v}$ on $b^p_\alpha$. Then by  Proposition \ref{embedding-inc}
\begin{equation*}
 |\mathcal{L}_{v}(u)| \leq \|v\|_{b^{\infty}_{\beta}}\int_{\mathbb{B}} |I^t_s u|  \ d\nu_{\rho}  \lesssim  \|v\|_{b^{\infty}_{\beta}}\|u\|_{b_{\alpha}^{p}}.
\end{equation*}
This gives that $\|\mathcal{L}_{v}\|\lesssim \|v\|_{b^{\infty}_{\beta}}$ .

Conversely, let $\mathcal{L}\in (b^p_\alpha)'$. We will show that there exists $v\in b^{\infty}_\beta$ such that $\mathcal{L}(u)=\langle u, v \rangle$. Let $u \in b^p_\alpha$.
By Proposition \ref{embedding-inc}, we have $u \in b^1_\rho$. Since $\rho+1<s+1$ and $\rho+t>-1$ hold by (\ref{Dual-S}) and (\ref{Dual-T}), applying the integral representation (\ref{Reproducing}) gives  that
\begin{equation*}
u(x)=\frac{V_s}{V_{s+t}}\int_{\mathbb{B}} R_s(x,y) I^t_s u(y) \, d\nu_s(y).
\end{equation*}
We claim that
\begin{equation}\label{operator-L}
\mathcal{L}(u)=\frac{V_s}{V_{s+t}}\int_{\mathbb{B}}  \mathcal{L}_{x}\big(R_s(x,y)\big) I^t_s u(y) \, d\nu_s(y) .
\end{equation}
Here the subindex in $\mathcal{L}_{x}$ indicates the variable of the function with respect to which $\mathcal{L}$ operates. In case $u=Y_{k}^{j}$,
(\ref{operator-L}) is verified in the same way as \cite[Proof of Theorem 6.1]{R}. By linearity (\ref{operator-L}) holds for every harmonic polynomial. Let now $u\in b_{\alpha}^{p}$ be arbitrary. By Corollary \ref{Corollary-Basic Properties} (b), there exists a sequence $(u_{n})$ of harmonic polynomials such that $\|u_{n}-u\|_{b^{p}_{\alpha}}=\|I^t_s u_{n}-I^t_s u\|_{L^{p}_{\alpha}}\to 0$. Then
\begin{equation*}
\mathcal{L}(u)= \lim_{n \to \infty} \mathcal{L}(u_{n})= \lim_{n \to \infty}\int_{\mathbb{B}}  \mathcal{L}_{x}\big(R_s(x,y)\big) I^t_s u_{n}(y) \, d\nu_s(y)
\end{equation*}
and all we need to show is
\begin{equation}\label{operator-L1}
\int_{\mathbb{B}}  \mathcal{L}_{x}\big(R_s(x,y)\big) \left(I^t_s u_{n}(y)- I^t_s u(y)\right) \, d\nu_s(y)\to 0 \quad \text{as} \quad n\to\infty .
\end{equation}
Now, $\mathcal{L}_{x}\big(R_s(x,y)\big)$ is harmonic function of $y$ and by Lemma \ref{norm-kernel}
\begin{equation}\label{operator-L2}
|\mathcal{L}_{x}\big(R_s(x,y)\big) | \leq \|\mathcal{L}\|\|R_s(\cdot,y)\|_{b^{p}_{\alpha}}\lesssim \frac{\|\mathcal{L}\|}{(1-|y|^{2})^{(n+s)-(n+\alpha)/p}}.
\end{equation}
Using Lemma \ref{Lemma-Special} (with $I^t_s u(y)=(1-|y|^{2})^{t}D^t_s u(y)$, $D^t_s u$ being harmonic) and (\ref{operator-L2}) we deduce
\begin{align*}
&\int_{\mathbb{B}} | \mathcal{L}_{x}\big(R_s(x,y)\big)| \left|I^t_s u_{n}(y)- I^t_s u(y)\right| \, d\nu_s(y)\\
&\lesssim \int_{\mathbb{B}} | \mathcal{L}_{x}\big(R_s(x,y)\big)|^{p} \left|I^t_s u_{n}(y)- I^t_s u(y)\right|^{p} (1-|y|^{2})^{p(n+s)-n}\, d\nu(y)\\
&\lesssim \int_{\mathbb{B}} \left|I^t_s u_{n}(y)- I^t_s u(y)\right|^{p} (1-|y|^{2})^{\alpha}\, d\nu(y)\\
&=\|I^t_s u_{n}- I^t_s u\|_{L^{p}_{\alpha}}\to 0 \quad \text{as} \quad n\to\infty .
\end{align*}
Thus (\ref{operator-L1}) holds and our claim is verified.

Define
\begin{equation*}
v(y)=\frac{V_{s-t'}}{V_{s+t}} D^{-t'}_{s+t'} \overline{\mathcal{L}_{x}\big(R_s(x,y)\big)}
\end{equation*}
By (\ref{inverse of Dst})
$$I^{t'}_{s'}v =(1-|y|^{2})^{t'}D^{t'}_{s'}v =\frac{V_{s-t'}}{V_{s+t}}(1-|y|^{2})^{t'}\overline{\mathcal{L}_{x}\big(R_s(x,y)\big)}.$$ We conclude
\begin{equation*}
\mathcal{L}(u)=\int_{\mathbb{B}}  I^{t}_{s}u(y)\overline{I^{t'}_{s'}v(y)} \, d\nu_{\rho+\beta}(y).
\end{equation*}
Finally, it follows from (\ref{operator-L2}) that
\begin{equation*}
|I^{t'}_{s'}v(y)|\lesssim \frac{\|\mathcal{L}\|}{(1-|y|^{2})^{\beta}},
\end{equation*}
which shows $v \in b^\infty_\beta$ and $\|v\|_{b^{\infty}_\beta} \lesssim \| \mathcal{L} \|$.

\end{proof}

\subsection{Atomic Decomposition}\label{subsection-Atomic}
In this subsection we will prove Theorem \ref{atomicbesov}.

When $\alpha=0$ and $0<p<1$, Coifman-Rochberg theorem states that there exists a sequence $(x_m)$ of points of $\mathbb{B}$ with the following property: Let $s>n(1/p-1)$.

\begin{enumerate}
    \item[(i)] For every  $u\in b^{p}_{0}$, there exists $(\lambda_{m})\in \ell^{p}$ such that
\begin{equation}\label{atbesov}
u(x)=\sum_{m=1}^{\infty} \lambda_{m}(1-|x_m|^{2})^{n+s-n/p} R_{s}(x,x_m)
\end{equation}
and  $\|\lambda_{m}\|_{\ell^{p}}\lesssim \|u\|_{b^{p}_{0}}$.
    \item[(ii)]  For every $(\lambda_{m})\in \ell^{p}$, the function $u$ defined in (\ref{atbesov}) is in $b^{p}_{0}$ and
$ \|u\|_{b^{p}_{0}} \lesssim\|\lambda_{m}\|_{\ell^{p}}$.
  \end{enumerate}

 By using Proposition \ref{Proposition-The map Dst} we extend this theorem to all $\alpha\in \mathbb{R}$.
\begin{proof}[Proof of Theorem \ref{atomicbesov}]
We begin with proving part (ii). We first show that the series in  (\ref{atbesove}) converges absolutely and uniformly on compact subsets of $\mathbb{B}$. If $K \subset \mathbb{B}$ is compact, then $R_{s}(x,x_m)\lesssim 1$, $\forall x \in K$ by Lemma \ref{Lemma-Kernel-Estimate}. Also, since
$n+s-(n+\alpha)/p>0$, $(1-|x_m|^{2})^{n+s-(n+\alpha)/p}\leq 1$. Therefore, for $ x \in K$,
 \begin{equation}\label{atomicbesov-eq1}
\sum_{m=1}^{\infty} |\lambda_{m}|(1-|x_m|^{2})^{n+s-(n+\alpha)/p} |R_{s}(x,x_m)|\lesssim \sum_{m=1}^{\infty} |\lambda_{m}|\leq \|\lambda_{m}\|_{\ell^{1}}
\end{equation}
and $u$ defined in (\ref{atbesove}) is in $h(\mathbb{B})$. To see that $u\in b^{p}_{\alpha}$, pick $t$ such that $\alpha+pt>-1$. By uniform
convergence on compact subsets and Lemma \ref{Lemma-cont-Dst}, we can push $D_{s}^{t}$ into the series and using (\ref{Dst - Rs}) obtain
 \begin{equation*}
D_{s}^{t}u(x)=\sum_{m=1}^{\infty} \lambda_{m}(1-|x_m|^{2})^{n+s-(n+\alpha)/p} R_{s+t}(x,x_m).
\end{equation*}
Using that $0< p<1$,
\begin{align*}
\|u\|^{p}_{b^{p}_{\alpha}}&\sim \|D_{s}^{t}u\|^{p}_{b^{p}_{\alpha+pt}}=\int_{\mathbb{B}}\left|\sum_{m=1}^{\infty} \lambda_{m}(1-|x_{m}|^{2})^{n+s-(n+\alpha)/p} R_{s+t}(x,x_m)\right|^{p} d\nu_{\alpha+pt}(x)\\
&\lesssim \sum_{m=1}^{\infty} |\lambda_{m}|^{p}(1-|x_m|^{2})^{p(n+s)-(n+\alpha)}\int_{\mathbb{B}} |R_{s+t}(x,x_m)|^{p}(1-|x|^{2})^{\alpha+pt}
 d\nu(x).
\end{align*}
Estimating the  integral by Lemma \ref{norm-kernel} (with $p(n+s+t)-(n+\alpha+pt)>0$) , we obtain
$\|u\|_{b^{p}_{\alpha}}\lesssim \|\lambda_{m}\|_{\ell^{p}}$.

We now prove part (i). Let $(x_m)$ be a sequence as asserted by Coifman-Rochberg theorem. By Proposition \ref{Proposition-The map Dst},
$D_{s}^{-\alpha/p}u\in b^{p}_{0}$ and $\|D_{s}^{-\alpha/p}u\|_{b^{p}_{0}}\sim \|u\|_{b^{p}_{\alpha}}$. We apply Coifman-Rochberg theorem for
$b^{p}_{0}$, replacing $s$ with $s-\alpha/p$. There exists $(\lambda_{m})\in\ell^{p}$ such that
\begin{equation*}
D_{s}^{-\alpha/p}u(x)=\sum_{m=1}^{\infty} \lambda_{m}(1-|x_m|^{2})^{n+s-(n+\alpha)/p} R_{s-\alpha/p}(x,x_m)
\end{equation*}
and $\|\lambda_{m}\|_{\ell^{p}}\lesssim \|D_{s}^{-\alpha/p}u\|_{b^{p}_{0}}\sim \|u\|_{b^{p}_{\alpha}}$.

Apply $D_{s-\alpha/p}^{\alpha/p}$ to both sides. By (\ref{inverse of Dst}) on the left side we get $u(x)$. On the right, as in the proof of part (ii),
by uniform convergence on compact subsets we can push $D_{s-\alpha/p}^{\alpha/p}$ into the series and by (\ref{Dst - Rs}) this leads to
\begin{equation*}
u(x)=\sum_{m=1}^{\infty} \lambda_{m}(1-|x_m|^{2})^{n+s-(n+\alpha)/p} R_{s}(x,x_m).
\end{equation*}
This finishes the proof.
\end{proof}

\bibliographystyle{amsalpha}

\begin{thebibliography}{10}
%
\bibitem{Ahl}
L.V. Ahlfors,
\textit{M\"obius Transformations in Several Variables},
University of Minnesota, Minneapolis, 1981.
%
\bibitem{ABR}
S. Axler, P. Bourdon, \& W. Ramey,
\textit{Harmonic function theory},
2\textit{nd} ed., Grad. Texts in Math., vol. 137, Springer, New York, 2001.
%
\bibitem{CKY}
B. R. Choe, H. Koo \& H. Yi,
\textit{Derivatives of harmonic Bergman and Bloch functions on the Ball},
J. Math. Anal. Appl. \textbf{260} (2001), 100--123.
%
\bibitem{CR}
R. R. Coifman \& R. Rochberg,
\textit{Representation theorems for holomorphic and harmonic functions in $L^p$},
Ast\'erisque \textbf{77} (1980), 11--66.
%
\bibitem{DS}
A. E. Djrbashian \& F. A. Shamoian,
\textit{Topics in the theory of $A^p_\alpha$ spaces},
Teubner Texts in Mathematics, 105, BSB B. G. Teubner Verlagsgesellschaft, Leipzig, 1988.
%
\bibitem{DU2}
\"{O}. F. Do\u{g}an \& A. E. \"Ureyen,
\textit{Inclusion relations between harmonic
Bergman-Besov and weighted Bloch spaces on the
unit ball}, Czech. Math. J., to appear.
%
\bibitem{DU1}
\"O. F. Do\u gan \& A. E. \"Ureyen,
\textit{Weighted harmonic Bloch spaces on the ball},
Complex Anal. Oper. Theory, \textbf{12(5)} (2018), 1143--1177.


%
\bibitem{Fefferman}
C. Fefferman \& E. M. Stein,
\textit{$H^{p}$ spaces of several variables},
Acta Math.  \textbf{129} (1972), 137--193.
%
\bibitem{GKU1}
S. Gerg\"un, H. T. Kaptano\u glu, \& A. E. \"Ureyen,
\textit{Reproducing kernels for harmonic Besov spaces on the ball},
C. R. Math. Acad. Sci. Paris \textbf{347} (2009), 735--738.
%
\bibitem{GKU2}
S. Gerg\"{u}n, H. T. Kaptano\u glu \& A. E. \"Ureyen,
\textit{Harmonic Besov spaces on the ball},
Int. J. Math. \textbf{27} (2016), no.9, 1650070, 59 pp.
%
\bibitem{JP}
M. Jevti\'c \& M. Pavlovi\'c,
\textit{Harmonic Bergman functions on the unit ball in $\mathbb R^n$},
Acta Math. Hungar. \textbf{85} (1999), 81--96.
%
\bibitem{Kuran}
\"{U}. Kuran,
\textit{Subharmonic behaviour of $|h|^{p}$ ($p>0$, $h$ harmonic)},
J. London Math. Soc. \textbf{8} (1974), 529--538.

%
\bibitem{LS}
C. W. Liu \& J. H. Shi,
\textit{Invariant mean-value property and $\mathcal M$-harmonicity in the unit ball of $\mathbb R^n$},
Acta Math. Sin. \textbf{19} (2003), 187--200.
%
\bibitem{M}
J. Miao,
\textit{Reproducing kernels for harmonic Bergman spaces of the unit ball},
Monatsh. Math. \textbf{125} (1998), 25--35.
%
\bibitem{Pavlovic}
 M. Pavlovi\'c,
\textit{On subharmonic behaviour and oscillation of functions on balls  in $\mathbb R^n$},
Publ. Inst. Math. (N.S.) \textbf{55(69)} (1994), 18--22.
%
\bibitem{R}
G. Ren,
\textit{Harmonic Bergman spaces with small exponents in the unit ball},
Collect. Math. \textbf{53} (2003), 83--98.
%
\bibitem{ZZ}
R. Zhao \& K. Zhu,
\textit{Theory of Bergman spaces in the unit ball of $\mathbb{C}^n$},
M\'em. Soc. Math. Fr. \textbf{115} (2008), 103 pp.
%
\bibitem{KHZ}
K. Zhu,
\textit{Spaces of holomorphic Functions in the Unit Ball}, Graduate Texts in Mathematics,
Vol. 226, Springer, New York, 2005.
\end{thebibliography}

\end{document}